\documentclass[a4paper,10pt,reqno]{amsart}

\usepackage{amsthm,amsmath,amssymb}
\usepackage{mathrsfs}	
\usepackage{amsrefs}	
\usepackage{amsfonts}
\usepackage[T1]{fontenc}
\usepackage[all]{xy}
\usepackage{enumitem}
\usepackage{caption}
\usepackage{graphicx}
\usepackage[breaklinks]{hyperref} 
\usepackage{comment}

\usepackage{pgf,tikz}
\usepackage{color}
\usetikzlibrary{arrows,decorations.pathmorphing}

\definecolor{linkcolor}{rgb}{0,0,0.675}%
\hypersetup{colorlinks=true,linkcolor=linkcolor,citecolor=linkcolor,urlcolor=linkcolor,bookmarksdepth=2}%
\newtheorem{theorem}{Theorem}[section]
\newtheorem{lemma}[theorem]{Lemma}
\newtheorem{corollary}[theorem]{Corollary}
\newtheorem{proposition}[theorem]{Proposition}
\theoremstyle{definition}
\newtheorem{remark}[theorem]{Remark}
\newtheorem{definition}[theorem]{Definition}
\newtheorem{example}[theorem]{Example}
\newtheorem{dictionary}[theorem]{Dictionary}
\newtheorem{recipe}[theorem]{Recipe}
\newtheorem{summary}[theorem]{Summary}
\setitemize{topsep=0pt plus 3pt,itemsep=0pt,leftmargin=25pt,listparindent=\parindent}
\setenumerate{topsep=0pt plus 3pt,itemsep=-3pt,leftmargin=25pt,listparindent=\parindent}
\setenumerate[1]{label=\normalfont(\alph{enumi})}
\setenumerate[2]{label=\normalfont(\arabic{enumii})}
\setdescription{itemsep=0pt,leftmargin=25pt,listparindent=\parindent}

\newcommand*{\A}{\mathbb A}
\renewcommand*{\P}{\mathbb P}
\newcommand*{\Spec}{\operatorname{Spec}}

\newcommand*{\Hom}{\operatorname{Hom}}
\newcommand*{\End}{\operatorname{End}}

\newcommand*{\varEnd}{\mathit{\mathcal E\hskip-1pt{}nd}}
\newcommand*{\Ext}{\operatorname{Ext}}
\newcommand*{\varExt}{\mathit{\mathcal E\hskip-1pt{}xt}}

\newcommand*{\Cl}{\operatorname{Cl}}
\newcommand*{\Pic}{\operatorname{Pic}}

\newcommand*{\D}{\operatorname{D}}

\renewcommand*{\mod}{\ \operatorname{mod\,}}

\newcommand*{\Z}{\mathbb{Z}}

\newcommand*{\Q}{\mathbb{Q}}
\newcommand*{\C}{\mathbb{C}}

\newcommand*{\gen}{{\mathrm{g}}} 
\let\op\operatorname
\title[\scalebox{0.88}{Exceptional collections on Dolgachev surfaces associated with degenerations}]{Exceptional collections on Dolgachev surfaces associated with degenerations}
\author{Yonghwa Cho and Yongnam Lee}
\date{}

\address{Department of Mathematical Sciences, KAIST, 291 Daehak-ro, Yuseong-gu, Daejeon 305-701, Korea}

\email{yonghwa.cho@kaist.ac.kr}

\address{Department of Mathematical Sciences, KAIST, 291 Daehak-ro, Yuseong-gu, Daejeon 305-701, Korea}

\email{ynlee@kaist.ac.kr}



\begin{document}

\begin{abstract}
		Dolgachev surfaces are simply connected minimal elliptic surfaces with $p_g=q=0$ and of Kodaira dimension 1. These surfaces are constructed by logarithmic transformations of rational elliptic surfaces. In this paper, we explain the construction of Dolgachev surfaces via $\Q$-Gorenstein smoothing of singular rational surfaces with two cyclic quotient singularities. This construction is based on the paper\,\cite{LeePark:SimplyConnected}. Also, some exceptional bundles on Dolgachev surfaces associated with $\Q$-Gorenstein smoothing have been constructed based on the idea of Hacking\,\cite{Hacking:ExceptionalVectorBundle}. In the case if Dolgachev surfaces were of type $(2,3)$, we describe the Picard group and present an exceptional collection of maximal length. Finally, we prove that the presented exceptional collection is not full, hence there exists a nontrivial phantom category in the derived category.\par

\end{abstract}

	\maketitle

	\setcounter{tocdepth}{1}\tableofcontents
	\section {Introduction}
		
		In the last few decades, the derived category $\D^{\rm b}(S)$ of a nonsingular projective variety $S$ has been extensively studied by algebraic geometers. One of the attempts is to find an exceptional collection that is a sequence of objects $E_1,\ldots,E_n$ such that
		\[
			\Ext^k(E_i,E_j) = \left\{
				\begin{array}{cl}
					0 & \text{if}\ i > j\\
					0 & \text{if}\ i=j\ \text{and}\ k\neq 0 \\
					\C & \text{if}\ i=j\ \text{and}\ k=0.
				\end{array}
			\right.
		\]
		There were many approaches to find exceptional collections of maximal length if $S$ is a nonsingular projective surface with $p_g = q=0$. Gorodentsev and Rudakov\,\cite{GorodenstevRudakov:ExceptionalBundleOnPlane} have classified all possible exceptional collections in the case $S = \P^2$, and exceptional collections on del Pezzo surfaces has been studied by Kuleshov and Orlov\,\cite{KuleshovOrlov:ExceptionalSheavesonDelPezzo}. For Enriques surfaces, Zube \cite{Zube:ExceptionalOnEnriques} gives an exceptional collection of length $10$, and the orthogonal part is studied by Ingalls and Kuznetsov\,\cite{IngallsKuznetsov:EnriquesQuarticDblSolid} for nodal Enriques surfaces. After initiated by the work of B\"ohning, Graf von Bothmer, and Sosna\,\cite{BGvBS:ExeceptCollec_Godeaux}, there also have come numerous results on the surfaces of general type\,({\it e.g.} \cite{GalkinShinder:Beauville,BGvBKS:DeterminantalBarlowAndPhantom,AlexeevOrlov:DerivedOfBurniat,Coughlan:ExceptionalCollectionOfGeneralType,KSLee:Isogenus_1,GalkinKatzarkovMellitShinder:KeumFakeProjective,Keum:FakeProjectivePlanes}). 
		For surfaces with Kodaira dimension one, such exceptional collections have not been shown to exist, thus it is a natural attempt to find an exceptional collection in $\D^{\rm b}(S)$. In this paper, we use the technique of $\Q$-Gorenstein smoothing to study the case $\kappa(S) = 1$. As far as the authors know, this is the first time to establish an exceptional collection of maximal length on a surface with Kodaira dimension one.
		
		The key ingredient is the method of Hacking\,\cite{Hacking:ExceptionalVectorBundle}, which associates a $T_1$-singularity $(P \in X)$ with an exceptional vector bundle on the general fiber of a $\Q$-Gorenstein smoothing of $X$. A $T_1$-singularity is the cyclic quotient singularity
		\[
			(0 \in \A^2 \big/ \langle \xi \rangle),\quad \xi \cdot(x,y) = (\xi x, \xi^{na-1}y),
		\]
		where $n > a > 0$ are coprime integers and $\xi$ is the primitive $n^2$-th root of unity\,(see the works of Koll\'ar and Shepherd-Barron\,\cite{KSB:CompactModuliOfSurfaces}, Manetti\,\cite{Manetti:NormalDegenerationOfPlane}, and Wahl\,\cite{Wahl:EllipticDeform,Wahl:SmoothingsOfNormalSurfaceSings} for the classification of $T_1$-singularities and their smoothings). In the paper\,\cite{LeePark:SimplyConnected}, Lee and Park constructed new surfaces of general type via $\Q$-Gorenstein smoothings of projective normal surfaces with $T_1$-singularities. Motivated by \cite{LeePark:SimplyConnected}, substantial amount of works was carried out, especially on (1) construction of new surfaces of general type\,({\it e.g.}\,\cite{KeumLeePark:GeneralTypeFromElliptic,LeeNakayama:SimplyGenType_PositiveChar,ParkParkShin:SimplyConnectedGenType_K3,ParkParkShin:SimplyConnectedGenType_K4}); (2) investigation of the KSBA boundaries of the moduli of space of surfaces of general type\,({\it e.g.} \cite{HackingTevelevUrzua:FlipSurfaces,Urzua:IdentifyingNeighbors}). Our approach is based on rather different perspective:
		\begin{center}
			Construct $S$ via a smoothing of a singular surface as in \cite{LeePark:SimplyConnected}, and apply \cite{Hacking:ExceptionalVectorBundle} to investigate $\Pic S$.
		\end{center}
		We study the case $S={}$a Dolgachev surface with two multiple fibers of multiplicities $2$ and $3$, and give an explicit $\Z$-basis for the N\'eron-Severi lattice of $S$\,(Theorem~\ref{thm:Synop_NSLattice}). Afterwards, we find an exceptional collection of line bundles of maximal length in $\D^{\rm b}(S)$\,(Theorem~\ref{thm:Synop_ExceptCollection_MaxLength}).
		\subsection*{Notations and Conventions}
			Throughout this paper, everything will be defined over the field of complex numbers. A surface is an irreducible projective variety of dimension two. 
			If $T$ is a scheme of finite type over $\C$ and $t \in T$ a closed point, then we use $(t \in T)$ to indicate the analytic germ.
			
		Let $n > a > 0$ be coprime integers, and let $\xi$ be the $n^2$-th root of unity. The $T_1$-singularity
		\[
			( 0 \in \A^2 \big/ \langle \xi \rangle ),\quad \xi\cdot(x,y) = (\xi x , \xi^{na-1}y)
		\]
		will be denoted by $\bigl( 0 \in \A^2 \big/ \frac{1}{n^2}(1,na-1) \bigr)$.
		
		If two divisors $D_1$ and $D_2$ are linearly equivalent, we write $D_1 = D_2$ if there is no ambiguity. Two $\Q$-Cartier Weil divisors $D_1,D_2$ are $\Q$-linearly equivalent, denoted by $D_1 \equiv D_2$, if there exists $r \in \Z_{>0}$ such that $rD_1 = rD_2$.
		
		Let $S$ be a nonsingular projective variety. The following invariants are associated with $S$.
		\begin{itemize}[fullwidth,itemindent=10pt]
			\item The geometric genus $p_g(S) = h^2(\mathcal O_S)$.
			\item The irregularity $q(S) = h^1(\mathcal O_S)$.
			\item The holomorphic Euler characteristic $\chi(S)$.
			\item The N\'eron-Severi group $\op{NS}(S) = \Pic S / \Pic^0 S$, where $\Pic^0 S$ is the group of divisors algebraically equivalent to zero.
		\end{itemize}
		
		Since the definitions of Dolgachev surfaces vary in literature, we fix our definition.
		\begin{definition}
			Let $q > p > 0$ be coprime integers. A \emph{Dolgachev surface $S$ of type $(p,q)$} is a minimal, simply connected, nonsingular, projective surface with $p_g(S) = q(S) = 0$ and of Kodaira dimension one such that there are exactly two multiple fibers of multiplicities $p$ and $q$.
		\end{definition}
		
		In the sequel, we will be given a degeneration $S \rightsquigarrow X$ from a nonsingular projective surface $S$ to a projective normal surface $X$, and compare information between them. We use the superscript ``$\gen$'' to emphasize this correlation. For example, we use $X^\gen$ instead of $S$. 
		\subsection*{Synopsis of the paper} In Section~\ref{sec:Construction}, we construct a Dolgachev surface $X^\gen$ of type $(2,n)$ following the technique of Lee and Park\,\cite{LeePark:SimplyConnected}. We begin with a pencil of plane cubics generated by two general nodal cubics, which meet at nine different points. The pencil defines a rational map $\P^2 \dashrightarrow \P^1$, undefined at the nine points of intersection. Blowing up the nine intersection points resolves the indeterminacy of $\P^2 \dashrightarrow \P^1$, hence yields a rational elliptic surface. After additional blow ups, we get two special fibers
		\[
			F_1 := C_1 \cup E_1,\quad\text{and}\quad F_2:= C_2\cup E_2\cup \ldots \cup E_{r+1}.
		\]
		Let $Y$ denote the resulting rational elliptic surface with the general fiber $C_0$, and let $p \colon Y \to \P^2$ denote the blow down morphism. Contracting the curves in the $F_1$ fiber\,(resp. $F_2$ fiber) except $E_1$\,(resp. $E_{r+1}$), we get the morphism $\pi \colon Y \to X$ to a projective normal surface $X$ with two $T_1$-singularities of types 
		\[
			(P_1 \in X) \simeq \Bigl( 0 \in \A^2 \Big/ \frac{1}{4}(1,1) \Bigr) \quad \text{and}\quad (P_2 \in X) \simeq \Bigl( 0 \in \A^2 \Big/ \frac{1}{n^2}(1,na-1) \Bigr)
		\]
		for coprime integers $n > a > 0$. Note that the numbers $n,a$ are determined by the formula
		\[
			\frac{n^2}{na-1} = k_1 - \frac{1}{	k_2 - \frac{1}{\ldots -\frac{1}{k_r}}	},
		\]
		where $-k_1,\ldots,-k_r$ are the self-intersection numbers of the curves in the chain $\{C_2,\ldots,E_r\}$\,(with the suitable order).
		We prove the formula\,(Proposition~\ref{prop:SingularSurfaceX})
		\begin{equation}
			\pi^* K_X \equiv - C_0 + \frac{1}{2}C_0 + \frac{n-1}{n}C_0,		\label{eq:Synop_QuasiCanoncialBdlFormula}
		\end{equation}
		which resembles the canonical bundle formula for minimal elliptic surfaces\,\cite[p.~213]{BHPVdV:Surfaces}. We then obtain $X^\gen$ by taking a general fiber of a $\Q$-Gorenstein smoothing of $X$. Then, since the divisor $\pi_* C_0$ is away from singularities of $X$, it moves to a nonsingular elliptic curve $C_0^\gen$ along the deformation $X \rightsquigarrow X^\gen$. We prove that the linear system $\lvert C_0^\gen \rvert$ defines an elliptic fibration $f^\gen \colon X^\gen \to \P^1$. Comparing (\ref{eq:Synop_QuasiCanoncialBdlFormula}) with the canonical bundle formula on $X^\gen$, we achieve the following theorem.
		\begin{theorem}[see Theorem~\ref{thm:SmoothingX} for details]\label{thm:Synop_NSLattice}
			Let $\varphi \colon \mathcal X \to (0 \in T)$ be a one parameter $\Q$-Gorenstein smoothing of $X$ over a smooth curve germ. Then for a general point $0 \neq t_0 \in T$, the fiber $X^\gen := \mathcal X_{t_0}$ is a Dolgachev surface of type $(2,n)$.
		\end{theorem}

		We jump into the case $a=1$ in Section~\ref{sec:ExcepBundleOnX^g}, and explain the construction of exceptional vector bundles\,(mostly line bundles) on $X^\gen$ associated with the degeneration $X^\gen \rightsquigarrow X$ using the method developed in \cite{Hacking:ExceptionalVectorBundle}. Let $\iota \colon Y \to \tilde X_0$ be the contraction of $E_2,\ldots,E_r$. Then, $Z_1 := \iota(C_1)$ and $Z_2 := \iota(C_2)$ are smooth rational curves. There exists a proper birational morphism $\Phi \colon \tilde{\mathcal X} \to \mathcal X$\,(a weighted blow up at the singularities of $X = \mathcal X_0$) such that the central fiber $\tilde{\mathcal X}_0 := \Phi^{-1}(\varphi^{-1}(0))$ is described as follows: it is the union of $\tilde X_0$, the projective plane $W_1 = \P^2_{x_1,y_1,z_1}$, and the weighted projective plane $W_2 = \P_{x_2,y_2,z_2}(1, n-1, 1)$ attached along
		\[
			Z_1 \simeq (x_1y_1=z_1^2) \subset W_1,\quad\text{and}\quad Z_2 \simeq (x_2y_2=z_2^n) \subset W_2.
		\]
		Intersection theory on $W_1$ and $W_2$ tells $\mathcal O_{W_1}(1)\big\vert_{Z_1} = \mathcal O_{Z_1}(2)$ and $\mathcal O_{W_2}(n-1)\big\vert_{Z_2} = \mathcal O_{Z_2}(n)$. The central fiber $\tilde{\mathcal X}_0$ has three irreducible components\,(disadvantage), but each component is more manageable than $X$\,(advantage). We work with the smoothing $\tilde{\mathcal X}/(0 \in T)$ instead of $\mathcal X / (0\in T)$. The general fiber of $\tilde{\mathcal X}/(0\in T)$ does not differ from $\mathcal X/(0\in T)$, hence it is the Dolgachev surface $X^\gen$. If $D$ is a divisor on $Y$ satisfying
		\begin{equation}
			(D.C_1)=2d_1 \in 2\Z,\ (D.C_2)=nd_2 \in n\Z, \text{ and } (D.E_2) = \ldots = (D.E_r) = 0, \label{eq:Synop_GoodDivisorOnY}
		\end{equation}		
		then there exists a line bundle $\tilde{\mathcal D}$ on $\tilde{\mathcal X}_0$ such that
		\[
			\tilde{\mathcal D}\big\vert_{\tilde X_0} \simeq \mathcal O_{\tilde X_0}(\iota_*D),\quad \tilde{\mathcal D}\big\vert_{W_1} \simeq \mathcal O_{W_1}(d_1),\quad \text{and}\quad \tilde{\mathcal D}\big\vert_{W_2} \simeq \mathcal O_{W_2}((n-1)d_2).
		\]
		It can be shown that the line bundle $\tilde{\mathcal D}$ is exceptional, hence it deforms uniquely to give a bundle $\mathscr D$ on the family $\tilde{\mathcal X}$. In this method, we construct $D^\gen \in\Pic X^\gen$ as the divisor associated with the line bundle $\mathscr D\big\vert_{X^\gen}$.
		
		There is a natural topological description of $D^\gen$. Let $B_i \subset X$ be a contractible ball around the singularity $P_i$ and let $M_i$ be the Milnor fiber associated to the smoothing $(P_i \in \mathcal X) / (0 \in T)$. Then $X^\gen$ is diffeomorphic to $(X \setminus (B_1 \cup B_2) ) \cup (M_1 \cup M_2)$, where the union is made by pasting along the natural diffeomorphism $\partial B_i \simeq \partial M_i$\,(see \cite[p.~39]{Manetti:ModuliOfDiffeo}). By Proposition~\ref{prop:Hacking_Specialization}, the relative homology sequence for the pair $(X,\, M_1 \cup M_2)$ reads
		\[
			0 \to H_2(X^\gen,\Z) \to H_2(X,\Z) \to H_1(M_1,\Z) \oplus H_1(M_2,\Z).
		\]
		Since $H_1(M_1,\Z) \simeq \Z/2\Z$ and $H_2(M_2,\Z) \simeq \Z/n\Z$, if $D \in \Pic Y$ is a divisor which fits into the condition (\ref{eq:Synop_GoodDivisorOnY}), then $[\pi_*D] \in H_2(X,\Z)$ maps to the zero element in $H_1(M_1, \Z) \oplus H_1(M_2,\Z)$. Thus, there exists a preimage of $[\pi_*D] \in H_2(X,\Z)$ along $H_2(X^\gen, \Z) \to H_2(X,\Z)$, which is nothing but the Poincar\'e dual of the first Chern class of $\mathcal O_{X^\gen}(D^\gen)$.
				
		Section~\ref{sec:NeronSeveri} concerns the case $n=3$ and $a=1$. Let $D$, $\tilde{\mathcal D}$ and $D^\gen$ be chosen as above. There exists a short exact sequence
		\begin{equation}
			0 \to \tilde{\mathcal D} \to \mathcal O_{\tilde X_0}(\iota_* D) \oplus \mathcal O_{W_1}(d_1) \oplus \mathcal O_{W_2}(2d_2) \to \mathcal O_{Z_1}(2d_1) \oplus \mathcal O_{Z_2}(3d_2) \to 0. \label{eq:Synop_CohomologySequence}
		\end{equation}
		This expresses $\chi(\tilde{\mathcal D})$ in terms of $\chi(\iota_*D)$, $d_1$, and $d_2$. Since the Euler characteristic is deformation invariant, we get $\chi(D^\gen) = \chi(\tilde{\mathcal D})$. Furthermore, it can be proved that $(C_0. D) = (C_0^\gen . D^\gen)$. This implies that $(C_0 . D) = (6 K_{X^\gen} . D^\gen)$. The Riemann-Roch formula reads
		\[
			(D^\gen)^2 = \frac{1}{6}(C_0. D) + 2 \chi(\tilde{\mathcal D}) - 2,
		\]
		which is a clue for discovering the N\'eron-Severi lattice $\op{NS}(X^\gen)$. This leads to the first main theorem of this paper:
		\begin{theorem}[$={}$Theorem~\ref{thm:Picard_ofGeneralFiber}]
			Let $H \in \Pic \P^2$ be the hyperplane divisor, and let $L_0 = p^*(2H)$. Consider the following correspondences of divisors\,(see Figure~\ref{fig:Configuration_Basic}).
			\[
				\begin{array}{c|c|c|c}
				\Pic Y & F_i - F_j & p^*H - 3F_9 & L_0 \\
				\hline%
				\Pic X^\gen & F_{ij}^\gen & (p^*H - 3F_9)^\gen & L_0^\gen \\[1pt]
				\end{array}\raisebox{-0.9\baselineskip}[0pt][0pt]{\,.}
			\]\vskip+5pt\noindent
			Define the divisors $\{G_i^\gen\}_{i=1}^{10} \subset \Pic X^\gen$ as follows:
			\begin{align*}
				G_i^\gen &= -L_0^\gen + 10K_{X^\gen} + F_{i9}^\gen,\quad i=1,\ldots,8;\\
				G_9^\gen &= -L_0^\gen + 11K_{X^\gen};\\
				G_{10}^\gen &= -3L_0^\gen + (p^*H - 3F_9)^\gen + 28K_{X^\gen}.
			\end{align*}
			Then the intersection matrix $\bigl( ( G_i^\gen . G_j^\gen) \bigr)$ is
			\[
				\left[
					\begin{array}{cccc}
						-1  & \cdots & 0 & 0 \\
						\vdots & \ddots & \vdots & \vdots \\
						0 &  \cdots & -1 & 0 \\
						0 &  \cdots & 0 & 1
					\end{array}			
				\right]\raisebox{-2\baselineskip}[0pt][0pt]{.}
			\]
			In particular, $\{G_i^\gen\}_{i=1}^{10}$ is a $\Z$-basis for the N\'eron-Severi lattice $\op{NS}(X^\gen)$.
		\end{theorem}
		\noindent We point out that the assumption $n=3$ is crucial for the definition of $G_{10}^\gen$. Indeed, its definition is motivated by the proof of \cite[Theorem~3.1]{Vial:Exceptional_NeronSeveriLattice}. The divisor $G_{10}^\gen$ has been chosen to satisfy
		\[
			K_{X^\gen} = G_1^\gen + \ldots + G_9^\gen - 3G_{10}^\gen,
		\]
		which does not make sense for $n>3$ as $K_{X^\gen}$ is not primitive.
		
		In Section~\ref{sec:ExcepCollectMaxLength} we continue to assume $n=3$, $a=1$. We give the proof of the second main theorem of the paper:
		\begin{theorem}[$={}$Theorem~\ref{thm:ExceptCollection_MaxLength} and Corollary~\ref{cor:Phantom}]\label{thm:Synop_ExceptCollection_MaxLength}
			Assume that $X^\gen$ is originated from a cubic pencil $\lvert \lambda p_*C_1 + \mu p_*C_2\rvert$ generated by two general nodal cubics. Then, there exists a semiorthogonal decomposition
			\[
				\bigr\langle \mathcal A,\ \mathcal O_{X^\gen},\ \mathcal O_{X^\gen}(G_1^\gen),\ \ldots,\ \mathcal O_{X^\gen}(G_{10}^\gen),\ \mathcal O_{X^\gen}(2G_{10}^\gen) \bigr\rangle
			\]
			of $\D^{\rm b}(X^\gen)$, where $\mathcal A$ is nontrivial phantom category\,({\it i.e.} $K_0(\mathcal A) = 0$, $\op{HH}_\bullet(\mathcal A) = 0$, but $\mathcal A\not\simeq 0$).
		\end{theorem}
		The proof contains numerous cohomology computations. As usual, the main ingredients which relate the cohomologies between $X$ and $X^\gen$ are the upper-semicontinuity and the invariance of Euler characteristics. The cohomology long exact sequence of (\ref{eq:Synop_CohomologySequence}) begins with
		\[
			0 \to H^0(\tilde{\mathcal D}) \to H^0(\iota_*D) \oplus H^0(\mathcal O_{W_1}(d_1)) \oplus H^0(\mathcal O_{W_2}(2d_2)) \to H^0(\mathcal O_{Z_1}(2d_1)) \oplus H^0(\mathcal O_{Z_2}(3d_2)).
		\]
		We prove that if $(D.C_1) = 2d_1 \leq 2$, $(D.C_2) = 3d_2 \leq 3$, and $(D.E_2)=0$, then $h^0(\tilde{\mathcal D}) \leq h^0(D)$. This gives an upper bound of $h^0(D^\gen)$. By Serre duality, $h^2(D^\gen) = h^0(K_{X^\gen} - D^\gen)$, hence we are able to use the same method to estimate the upper bound. After the computations of upper bounds of $h^0(D^\gen)$ and $h^2(D^\gen)$, the upper bound of $h^1(D^\gen)$ can be examined by looking at $\chi(D^\gen)$. For any divisor $D^\gen$ which appears in the proof of Theorem~\ref{thm:Synop_ExceptCollection_MaxLength}, at least one of $\{h^0(D^\gen), h^2(D^\gen)\}$ is zero, and the other one is bounded by $\chi(D^\gen)$. Then, $h^1(D^\gen)=0$ and all the three numbers $(h^p(D^\gen) : p=0,1,2)$ are exactly evaluated. One obstruction to this argument is the condition $d_1, d_2 \leq 1$, but it can be dealt with the following observation:
		\begin{center}
			if a line bundle on $X^\gen$ is obtained from either $C_1$ or $2C_2+E_2$, then it is trivial.
		\end{center}
		Perturbing $D$ by $C_1$ and $2C_2+E_2$, we can adjust the numbers $d_1$, $d_2$.
		
		The proof reduces to find a suitable upper bound of $h^0(D)$. One of the very first trials is to find a smooth rational curve $C \subset Y$ such that $(D.C)$ is small. Then, by the short exact sequence $0 \to \mathcal O_Y(D-C) \to \mathcal O_Y(D) \to \mathcal O_C(D) \to 0$, we get $h^0(D) \leq h^0(D-C) + \min\{ 0,\,(D.C)+1\}$. Replace $D$ by $D-C$ and repeat this procedure. It eventually stops when the value of $h^0(D-C)$ is understood immediately\,({\it e.g.} when $D-C$ is linearly equivalent to a negative sum of effective curves). This will give an upper bound of $h^0(D)$. This method sometimes gives a sharp bound of $h^0(D)$, but sometimes not. Indeed, some cohomologies depend on the configuration of generating cubics $p_*C_1$, $p_*C_2$ of the cubic pencil, while the previous numerical argument cannot capture the configuration of $p_*C_1$ and $p_*C_2$. For those cases, we find an upper bound of $h^0(D)$ as follows. Assume that $D$ is an effective divisor. Then, $p_*D \subset \P^2$ is a plane curve. The divisor form of $D$ determines the degree of $p_*D$ and some conditions that $p_*D$ must admit. For example, consider $D = p^*H - E_1$. The exceptional curve $E_1$ is obtained by blowing up the node of $p_*C_1$. Hence, $p_*D$ must be a line passing through the node of $p_*C_1$. In this way, conditions can be represented by an ideal $\mathcal I \subset \mathcal O_{\P^2}$. Hence, proving $h^0(D) \leq r$ reduces to proving $h^0(\mathcal O_{\P^2}\bigl(\deg p_*D) \otimes \mathcal I\bigr) \leq r$. The latter one can be computed via a computer-based approach\,(Macaulay2). Finally, $\mathcal A\not\simeq 0$ is guaranteed by the argument involving anticanonical pseudoheight due to Kuznetsov\,\cite{Kuznetsov:Height}.
		
		We remark that a (simply connected) Dolgachev surface of type $(2,n)$ cannot have an exceptional collection of maximal length for any $n > 3$ as explained in \cite[Theorem~3.10]{Vial:Exceptional_NeronSeveriLattice}. Also, Theorem~\ref{thm:Synop_ExceptCollection_MaxLength} gives an answer to the question posed in \cite[Remark~3.12]{Vial:Exceptional_NeronSeveriLattice}.
	\section{Construction of Dolgachev Surfaces}	\label{sec:Construction}
		Let $n$ be an odd integer. This section presents the construction of Dolgachev surfaces of type $(2,n)$. The construction follows the technique introduced in \cite{LeePark:SimplyConnected}. Let $C_1,C_2 \subseteq \P^2$ be general nodal cubic curves meeting at $9$ different points, and let $Y' = \op{Bl}_9\P^2 \to \P^2$ be the blow up at the intersection points. Then the cubic pencil $\lvert \lambda C_1 + \mu C_2\rvert$ defines an elliptic fibration $Y' \to \P^1$, with two special fibers $C_1'$ and $C_2'$ (which correspond to the proper transforms of $C_1$ and $C_2$, respectively). Blowing up the nodes of $C_1'$ and $C_2'$, we obtain the $(-1)$-curves, say $E_1$ and $E_2$ respectively. Also, blowing up one of the intersection points of $C_2''$\,(the proper transform of $C_2'$) and $E_2$, we obtain the configuration described in Figure~\ref{fig:Configuration_Basic}.
		\begin{figure}[h]
			\centering\scalebox{0.88}{
			\begin{tikzpicture}
				\draw(0,0)	node[anchor=center] {\includegraphics[scale=0.75, bb = -15 -15 334 240]{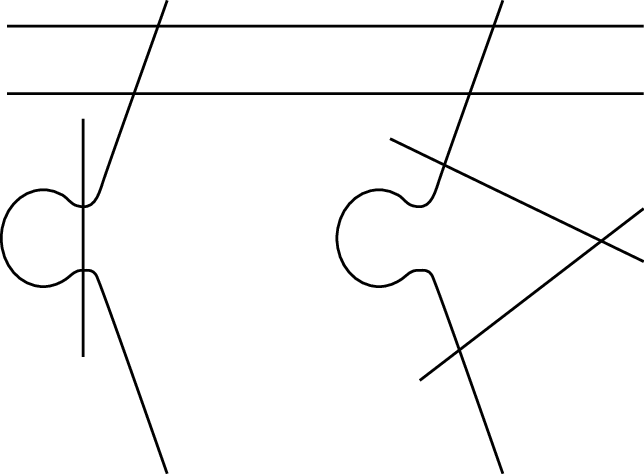}};
				\draw(0,0)	node[anchor=east] (C2) {$C_2''$};
				\node[below,shift=(90:2pt)] at (C2.south) {$\scriptstyle (-5)$};
				\draw(-4.25,0) node[anchor=east] (C1){$C_1''$};
				\node[below,shift=(90:2pt)]	at (C1.south) {$\scriptstyle (-4)$};
				\draw(-3.25,-1.4) node[anchor=north] (E1) {$E_1$};
				\node[below,shift=(90:4pt)] at (E1.south) {$\scriptstyle (-1)$};
				\draw(-0.5,2.75) node[anchor=south] (F1) {$F_1,\ldots,F_9$};
				\node[above,shift=(270:7pt)] at (F1.north) {$\scriptstyle (-1)$};
				\draw(-0.5,2.44) node[anchor=center]{$\vdots$};
				\draw(2.5,-1.1)	node[anchor=center] (E2) {$E_2'$};
				\node[below,shift=(90:2pt)] at (E2.south) {$\scriptstyle (-2)$};
				\draw(2.5,0.8)	node[anchor=center] (E3) {$E_3$};
				\node[above,shift=(270:3pt)] at (E3.north) {$\scriptstyle (-1)$};
			\end{tikzpicture}}\vskip-14pt%
			\caption{Configuration of the divisors in the surface obtained by blowing up two points of $Y'$.}\label{fig:Configuration_Basic}
		\end{figure}
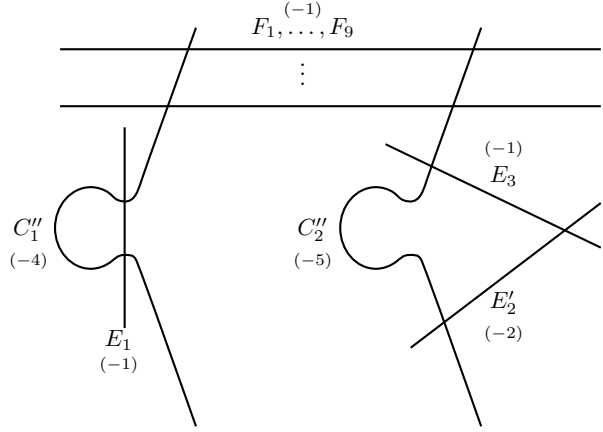
		The divisors $F_1,\ldots,F_9$ are the proper transforms of the exceptional fibers of $Y' = \op{Bl}_9\P^2 \to \P^2$. The numbers in the parentheses are self-intersection numbers of the corresponding divisors.
%
		On the fiber $C_2'' \cup E_2' \cup E_3$, we can think of two different blow ups as the following dual intersection graphs illustrate.
		\[
			\begin{tikzpicture}[scale=1]
				\draw(0,0) node[anchor=center] (C2) {};
				\draw(40pt,0pt) node[anchor=center] (E2) {};
				\draw(20pt,15pt) node[anchor=center] (E3) {};
				\node[below,shift=(90:1pt)] at (C2.south) {$\scriptstyle -5$};
				\node[below,shift=(90:1pt)] at (E2.south) {$\scriptstyle -2$};
				\node[above] at (E3.north) {$\scriptstyle -1$};
				\fill[red] (C2) circle (1.5pt);
				\fill[blue] (E2) circle (1.5pt);
				\fill[black] (E3) circle (1.5pt);
				\draw[red,-] (C2.north east) -- (E3.south west) node[above left, align=center, midway]{\tiny L};
				\draw[blue,-] (E2.north west) -- (E3.south east) node[above right, align=center, midway]{\tiny R};
				\draw[-] (C2.east) -- (E2.west);
				\begin{scope}[shift={(-170pt,0pt)}]
					\draw(0,0) node[anchor=center] (L C2) {};
					\draw(40pt,0pt) node[anchor=center] (L E2) {};
					\draw(40pt,15pt) node[anchor=center] (L E3) {};
					\draw(80pt,0pt) node[anchor=center] (L E4) {};
					\node[below] at (L C2.south) {$\scriptstyle -6$};
					\node[below] at (L E2.south) {$\scriptstyle -2$};
					\node[below] at (L E4.south) {$\scriptstyle -2$};
					\node[above] at (L E3.north) {$\scriptstyle -1$};
					\fill[red] (L C2) circle (1.5pt);
					\fill[blue] (L E2) circle (1.5pt);
					\fill[red] (L E3) circle (1.5pt);
					\fill[black] (L E4) circle (1.5pt);
					\draw[-] (L C2.north east) -- (L E3.south west) node[above, align=center, midway]{\tiny L'};
					\draw[-] (L E2.east) -- (L E4.west);
					\draw[-] (L C2.east) -- (L E2.west);
					\draw[-] (L E4.north west) -- (L E3.south east) node[above, align=center, midway]{\tiny R'};
				\end{scope}
				\begin{scope}[shift={(130pt,0pt)}]
					\draw(0,0) node[anchor=center] (R C2) {};
					\draw(40pt,0pt) node[anchor=center] (R E2) {};
					\draw(40pt,15pt) node[anchor=center] (R E3) {};
					\draw(80pt,0pt) node[anchor=center] (R E4) {};
					\node[below] at (R C2.south) {$\scriptstyle -2$};
					\node[below] at (R E2.south) {$\scriptstyle -5$};
					\node[below] at (R E4.south) {$\scriptstyle -3$};
					\node[above] at (R E3.north) {$\scriptstyle -1$};
					\fill[black] (R C2) circle (1.5pt);
					\fill[red] (R E2) circle (1.5pt);
					\fill[blue] (R E3) circle (1.5pt);
					\fill[blue] (R E4) circle (1.5pt);
					\draw[-] (R C2.north east) -- (R E3.south west) node[above, align=center, midway]{\tiny L'};
					\draw[-] (R E2.east) -- (R E4.west);
					\draw[-] (R C2.east) -- (R E2.west);
					\draw[-] (R E4.north west) -- (R E3.south east) node[above, align=center, midway]{\tiny R'};
				\end{scope}
				\draw [->,decorate,decoration={snake,amplitude=1pt,segment length=5pt, post length=2pt}]	(-15pt,5pt) -- (-75pt, 5pt) node[below, align=center, midway]{$\scriptstyle \op{Bl}_{\rm L}$};
				\draw [->,decorate,decoration={snake,amplitude=1pt,segment length=5pt, post length=2pt}]	(55pt,5pt) -- (115pt, 5pt) node[below, align=center, midway]{$\scriptstyle \op{Bl}_{\rm R}$};
			\end{tikzpicture}
		\]
		In general, if one has a fiber with configuration\!\!%
		\raisebox{-11pt}[0pt][13pt]{
			\begin{tikzpicture}
				\draw(0,0) node[anchor=center] (E1) {};
				\draw(30pt,0pt) node[anchor=center] (E2) {};
				\draw(60pt,0pt) node[anchor=center, inner sep=10pt] (E3) {};
				\draw(90pt,0pt) node[anchor=center] (E4) {};
				\fill[black] (E1) circle (1.5pt);
				\fill[black] (E2) circle (1.5pt);
				\draw (E3)	node[anchor=center]{$\cdots$};
				\fill[black] (E4) circle (1.5pt);
				\node[black,below,shift=(90:2pt)] at (E1.south) {$\scriptscriptstyle -k_1$};
				\node[black,below,shift=(90:2pt)] at (E2.south) {$\scriptscriptstyle -k_2$};
				\node[black,below,shift=(90:2pt)] at (E4.south) {$\scriptscriptstyle -k_r$};
				\draw[-] (E1.east) -- (E2.west);
				\draw[-] (E2.east) -- (E3.west);
				\draw[-] (E3.east) -- (E4.west);
			\end{tikzpicture} }\!\!\!\!, %
		then blowing up at L yields\!\!\!\!%
		\raisebox{-11pt}[0pt][13pt]{
			\begin{tikzpicture}
				\draw(0,0) node[anchor=center] (E1) {};
				\draw(25pt,0pt) node[anchor=center] (E2) {};
				\draw(50pt,0pt) node[anchor=center, inner sep=10pt] (E3) {};
				\draw(75pt,0pt) node[anchor=center] (E4) {};
				\draw(100pt,0pt) node[anchor=center] (E5) {};
				\fill[black] (E1) circle (1.5pt);
				\fill[black] (E2) circle (1.5pt);
				\draw (E3)	node[anchor=center]{$\cdots$};
				\fill[black] (E4) circle (1.5pt);
				\fill[black] (E5) circle (1.5pt);
				\node[black,below,shift=(90:2pt)] at (E1.south) {$\scriptscriptstyle -(k_1+1)$};
				\node[black,below,shift=(90:2pt)] at (E2.south) {$\scriptscriptstyle -k_2$};
				\node[black,below,shift=(90:2pt)] at (E4.south) {$\scriptscriptstyle -k_r$};
				\node[black,below,shift=(90:2pt)] at (E5.south) {$\scriptscriptstyle -2$};
				\draw[-] (E1.east) -- (E2.west);
				\draw[-] (E2.east) -- (E3.west);
				\draw[-] (E3.east) -- (E4.west);
				\draw[-] (E4.east) -- (E5.west);
		\end{tikzpicture} }\!\!. %
		Similarly, the blowing up at R yields\!\!%
		\raisebox{-12pt}[0pt][13pt]{
			\begin{tikzpicture}
				\draw(0,0) node[anchor=center] (E1) {};
				\draw(25pt,0pt) node[anchor=center] (E2) {};
				\draw(50pt,0pt) node[anchor=center, inner sep=10pt] (E3) {};
				\draw(75pt,0pt) node[anchor=center] (E4) {};
				\draw(108pt,0pt) node[anchor=center] (E5) {};
				\fill[black] (E1) circle (1.5pt);
				\fill[black] (E2) circle (1.5pt);
				\draw (E3)	node[anchor=center]{$\cdots$};
				\fill[black] (E4) circle (1.5pt);
				\fill[black] (E5) circle (1.5pt);
				\node[black,below,shift=(90:2pt)] at (E1.south) {$\scriptscriptstyle -2$};
				\node[black,below,shift=(90:2pt)] at (E2.south) {$\scriptscriptstyle -k_1$};
				\node[black,below,shift=(90:2pt)] at (E4.south) {$\scriptscriptstyle -k_{r-1}$};
				\node[black,below,shift=(90:2pt)] at (E5.south) {$\scriptscriptstyle -(k_r+1)$};
				\draw[-] (E1.east) -- (E2.west);
				\draw[-] (E2.east) -- (E3.west);
				\draw[-] (E3.east) -- (E4.west);
				\draw[-] (E4.east) -- (E5.west);
		\end{tikzpicture} }\!\!\!\!\!\!\!.\ \ %
		These present all possible resolution graphs of $T_1$-singularities\,\cite[Theorem~17]{Manetti:NormalDegenerationOfPlane}. Let $Y$ be the surface obtained after successive blow ups on the second special fiber $C_2'' \cup E_2' \cup E_3$, so that the resulting fiber contains the resolution graph of a $T_1$-singularity of type $\bigl(0 \in \A^2 / \frac{1}{n^2}(1, na-1)\bigr)$ for some odd integer $n$ and an integer $a$ with $\op{gcd}(n,a)=1$.
	
		To simplify notations, we will not distinguish the divisors and their proper transforms unless there arise ambiguities. For instance, the proper transform of $C_1 \in \Pic \P^2$ in $Y$ will be denoted by $C_1$, and so on. We fix this configuration of $Y$ throughout this paper, so it is appropriate to give a summary here:
		\begin{enumerate}[label=\normalfont(\arabic{enumi})]
			\item the $(-1)$-curves $F_1,\ldots,F_9$ that are proper transforms of the exceptional fibers of $\op{Bl}_9 \P^2 \to \P^2$;
			\item the $(-4)$-curve $C_1$ and the $(-1)$-curve $E_1$ arising from the blowing up of the first nodal curve;
			\item the negative curves $C_2,\,E_2,\,\ldots,\,E_r,\,E_{r+1}$, where $E_{r+1}^2 = -1$ and $C_2,\,E_2,\,\ldots,\,E_r$ form a resolution graph of a $T_1$-singularity of type $\bigl(0 \in \A^2 \big/\frac{1}{n^2}(1,na-1)\bigr)$.
		\end{enumerate}\vskip+5pt
		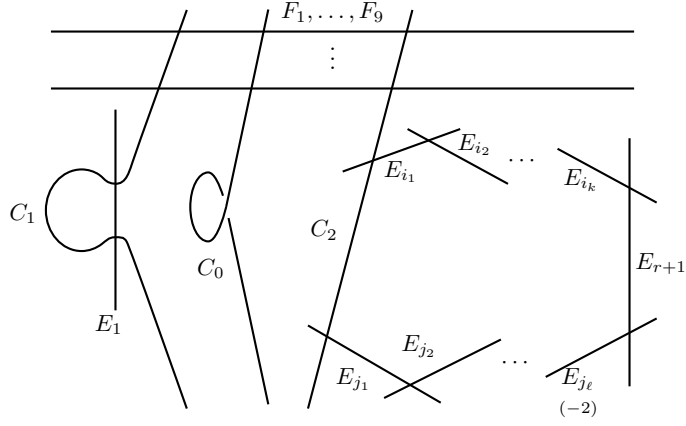
\begin{figure}[h]
			\centering\scalebox{0.88}{
				\begin{tikzpicture}
					\draw(0,0)	node[anchor=center] {\includegraphics[scale=0.75, bb = -17 -22 357.5 244.5]{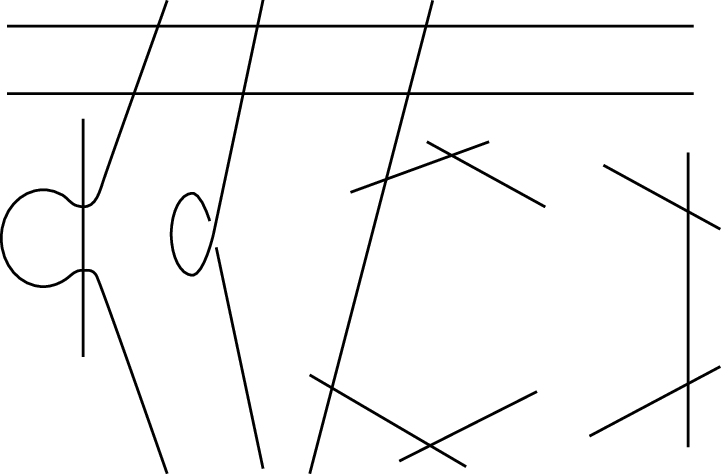}};
					\draw(0,-0.25)	node[anchor=east]  {$C_2$};
					\draw(-4.5,0) node[anchor=east] {$C_1$};
					\draw(-3.55,-1.4) node[anchor=north] {$E_1$};
					\draw(-0.2,2.75) node[anchor=south] {$F_1,\ldots,F_9$};
					\draw(-0.2,2.44) node[anchor=center]{$\vdots$};
					\draw(2.55,-2.25)node[anchor=center]  {$\cdots$};
					\draw(2.65,0.8)	node[anchor=center] {$\cdots$};
					\draw(0.5,-2.5)		node[anchor=east]	{$E_{j_1}$};
					\draw(0.7,-2)			node[anchor=west]	{$E_{j_2}$};
					\draw(3.85,-2.5)		node[anchor=east]	(Ejl) {$E_{j_\ell}$};
					\draw(1.2,0.92)	node[anchor=north east] {$E_{i_1}$};
					\draw(1.54,1.04)	node[anchor=west]	{$E_{i_2}$};
					\draw(3.9,0.80)	node[anchor=north east] {$E_{i_k}$};
					\draw(4.2,-0.75)	node[anchor=west]	{$E_{r+1}$};
					\draw(-2.35,-1.1)	node[anchor=south west] {$C_0$};
					\node[below,shift=(90:3pt)] at (Ejl.south) {$\scriptstyle(-2)$};
				\end{tikzpicture}	}\vskip-14pt
				\caption{Configuration of the surface $Y$. The sequence $E_{i_k},\,\ldots,\,E_{i_1},\,C_2,\allowbreak{}E_{j_1},\,\ldots,\,\allowbreak E_{j_\ell}$ forms the chain of the resolution graph of a $T_1$-singularity of type $\bigl(0 \in \A^2 \big/ \frac{1}{n^2}(1,na-1)\bigr)$. Without loss of generality, we may assume $j_\ell = r$. Note that $E_r^2 = -2$ by construction.
				}\label{fig:Configuration_General}
			\end{figure}
		Let $C_0$ be the general fiber of the elliptic fibration $Y \to \P^1$. The fibers are linearly equivalent, thus
		\begin{align}
			C_0 &= C_1 + 2E_1 \nonumber \\
			&= C_2 + a_2 E_2 + a_3 E_3 + \ldots + a_{r+1} E_{r+1}, \label{eq:SpecialFiber}
		\end{align}
		where $a_2,\ldots,a_{r+1}$ are the integers determined by the system of linear equations
		\begin{equation}\label{eq:EquationOnFiber}
			(C_2.E_i) + \sum_{j=2}^{r+1} a_j (E_j.E_i) = 0,\quad i=2,\ldots, r+1.
		\end{equation}
		Note that the values $(C_2.E_i)$, $(E_j.E_i)$ are explicitly given in the configuration\,(Figure~\ref{fig:Configuration_General}). The matrix $\bigl( (E_j.E_i) \bigr)_{2\leq i,j \leq r}$ is negative definite\,\cite{Mumford:TopologyOfNormalSurfaceSingularity}, and the number $a_{r+1}$ is determined by Proposition \ref{prop:SingIndexAndFiberCoefficients}, hence the system (\ref{eq:EquationOnFiber}) has a unique solution.
		\begin{lemma}\label{lem:CanonicalofY}
			In the above situation, the following formula holds:
			\[
				K_Y = E_1 - C_2 - E_2 - \ldots - E_{r+1}.
			\]
		\end{lemma}
		\begin{proof}
			The proof proceeds by an induction on $r$. The minimum value of $r$ is two, the case in which $C_2\cup E_2$ from the chain%
			\raisebox{0pt}[15pt][0pt]{
			\begin{tikzpicture}
				\draw(0,0) node[anchor=center] (E1) {};
				\draw(20pt,0pt) node[anchor=center] (E2) {};;
				\fill[black] (E1) circle (1.5pt);
				\fill[black] (E2) circle (1.5pt);
				\node[black,shift=(90:2pt)] at (E1.north) {$\scriptscriptstyle -5$};
				\node[black,shift=(90:2pt)] at (E2.north) {$\scriptscriptstyle -2$};
				\draw[-] (E1.east) -- (E2.west);
			\end{tikzpicture} }\!\!. %
			Let $H \in \Pic \P^2$ be a hyperplane divisor, and let $p \colon Y \to \P^2$ be the blowing down morphism. Then
			\[
				K_Y = p^* K_{\P^2} + F_1 + \ldots + F_9 + E_1 + d_2 E_2 + d_3E_3
			\]
			for some $d_2,d_3 \in \Z$. Since any cubic curve in $\P^2$ is linearly equivalent to $3 H$,
			\begin{align*}
				p^* ( 3H ) &= C_0 + F_1 + \ldots + F_9 \\
				&= (C_2 + a_2 E_2 +a_3 E_3) + F_1 + \ldots + F_9
			\end{align*}
			where $a_2,a_3$ are the integers introduced in (\ref{eq:SpecialFiber}). Hence, 
			\begin{align*}
				K_Y &= p^* (-3H) + F_1 + \ldots + F_9 + E_1 + d_2 E_2 + d_3 E_3 \\
				&= E_1 - C_2 + (d_2-a_2)E_2 + (d_3-a_3)E_3.
			\end{align*}
			Here, the genus formula shows that $K_Y = E_1 - C_2 - E_2 - E_3$.
			
			Assume the induction hypothesis that $K_Y = E_1 - C_2 - E_2 - \ldots - E_{r+1}$. Let $D \in \{C_2,E_2,\ldots,E_r\}$ be a divisor intersects $E_{r+1}$, and let $\varphi \colon \widetilde Y \to Y$ be the blowing up at the point $D \cap E_{r+1}$. Then,
			\[
				K_{\widetilde Y} = \varphi^* K_Y + \widetilde E_{r+2},
			\]
			where $\widetilde E_{r+2}$ is the exceptional divisor of $\varphi$. Let $\widetilde C_2, \widetilde E_1, \ldots, \widetilde E_{r+1}$ denote the proper transforms of the corresponding divisors. Then, $\varphi^*$ maps $D$ to $(\widetilde D + \widetilde E_{r+2})$, maps $E_{r+1}$ to $(\widetilde E_{r+1} + \widetilde E_{r+2})$, and maps the other divisors to their proper transforms. It follows that 
			\begin{align*}
				\varphi^*K_Y &= \varphi^*(E_1 - C_2 - \ldots - E_{r+1}) \\
				&= \widetilde E_1 - \widetilde C_2 - \ldots - \widetilde E_{r+1} - 2 \widetilde E_{r+2}.
			\end{align*}
			Hence, $K_{\widetilde Y} = \varphi^* K_Y + \widetilde E_{r+2} = \widetilde E_1 - \widetilde C_2 - \widetilde E_2 - \ldots - \widetilde E_{r+2}$.
		\end{proof}%
		\begin{proposition}\label{prop:SingularSurfaceX}
			Let $\pi \colon Y \to X$ be the contraction of the curves $C_1,\,C_2,\, E_2,\,\ldots,\, E_r$. Let $P_1 = \pi(C_1)$ and $P_2 = \pi(C_2 \cup E_2 \cup \ldots \cup E_r)$ be the singularities of types $\bigl( 0 \in \A^2 \big/ \frac{1}{4}(1,1)\bigr)$ and $\bigl( 0 \in \A^2 \big/ \frac{1}{n^2}(1,na-1)\bigr)$, respectively. Then the following properties of $X$ hold:
			\begin{enumerate}[ref=(\alph{enumi})]
				\item \label{item:SingularSurfaceX_Cohomologies}$X$ is a projective normal surface with $H^1(\mathcal O_X) = H^2(\mathcal O_X)=0$;
				\item $\pi^*K_X \equiv (\frac 12 - \frac 1n)C_0 \equiv  C_0 - \frac{1}{2} C_0 - \frac{1}{n} C_0$ as $\Q$-divisors.
			\end{enumerate}
			In particular, $K_X^2 = 0$, $K_X$ is nef, but $K_X$ is not numerically trivial.
		\end{proposition}
		\begin{proof}\ 
			\begin{enumerate}
				\item Since the singularities of $X$ are rational, $R^q \pi_* \mathcal O_Y = 0$ for $q > 0$. The Leray spectral sequence
				\[
					E_2^{p,q} = H^p( X, R^q\pi_* \mathcal O_Y ) \Rightarrow H^{p+q}(Y,\mathcal O_Y)
				\]
				says that $H^p(Y,\mathcal O_Y) \simeq H^p (X, \pi_* \mathcal O_Y) = H^p(X,\mathcal O_X)$ for $p > 0$. The surface $Y$ is obtained from $\P^2$ by a finite sequence of blow ups, hence $H^1(Y,\mathcal O_Y) = H^2(Y,\mathcal O_Y) =0$. One can immediately verify the hypotheses of Artin's criterion for contractibility\,\cite[Theorem~2.3]{Artin:Contractibility} hold, thus $X$ is projective.
				\item 			Since the morphism $\pi$ contracts $C_1,\,C_2,\,E_2,\,\ldots,\,E_r$, we may write
				\[
					\pi^* K_X \equiv K_Y + c_1 C_1 + c_2 C_2 + b_2 E_2 + \ldots + b_r E_r,
				\]
				for $c_1,c_2,b_2,\ldots,b_r \in \Q$(the coefficients may not be integral since $X$ is singular). It is easy to see that $c_1 = \frac 12$. By Lemma~\ref{lem:CanonicalofY},
				\[
					\pi^* K_X \equiv \frac{1}{2}C_0 + (c_2- 1)C_2 + (b_2 -1)E_2+ \ldots + (b_r-1) E_r - E_{r+1}.
				\]
				Both $\pi^*K_X$ and $C_0$ do not intersect with $C_2,E_2,\ldots,E_r$. Thus, we get
				\begin{equation}\label{eq:Aux1}
					\left\{
					\begin{array}{l@{}l}
						0 &{}= (1-c_2)(C_2^2) + \sum_{j =2}^r (1-b_j)(E_j.C_2)  + (E_{r+1}.C_2) \\
						0 &{}= (1-c_2)(C_2.E_i) + \sum_{j=2}^r (1-b_j)(E_j.E_i) + (E_{r+1}.E_i),\quad \text{for\ }i=2,\ldots,r.
					\end{array}
					\right.
				\end{equation}
				After divided by $a_{r+1}$, (\ref{eq:EquationOnFiber}) becomes
				\[
					0=\frac{1}{a_{r+1}} (C_2. E_i) + \sum_{j=2}^r \frac{a_j}{a_{r+1}} (E_j.E_i) + (E_{r+1}.E_i),\quad \text{for\ }i=2,\ldots,r.
				\]
				In addition, the equation $( C_2 + a_2 E_2 + \ldots + a_{r+1} E_r \mathbin. C_2 ) = (C_0 . C_2) = 0 $ gives rise to
				\[
					0=\frac{1}{a_{r+1}} (C_2^2) + \sum_{j=2}^r \frac{a_j}{a_{r+1}} (E_j.C_2) + (E_{r+1}.C_2).
				\]
				Comparing these equations with (\ref{eq:Aux1}), it is easy to see that the ordered tuples
				\[
					(1-c_2,\ 1-b_2,\ \ldots,\ 1-b_r)\quad\text{and}\quad (1/a_{r+1},\ a_2/a_{r+1},\ \ldots,\ a_r / a_{r+1})
				\]
				fit into the same system of linear equations. Since the intersection matrix of the divisors $(C_2,E_2,\ldots,E_r)$ is negative definite,
				\[
					(1-c_2,\, 1-b_2,\, \ldots,\, 1-b_r) = (1/a_{r+1},\ a_2/a_{r+1},\ \ldots,\ a_r / a_{r+1}).
				\]
				It follows that
				\begin{align*}
					\pi^* K_X %
					&\equiv \frac{1}{2}C_0 + (c_2 -1 )C_2 + (b_2 - 1)E_2 + \ldots + (b_r -1) E_r - E_{r+1} \\
					&\equiv \frac{1}{2}C_0 - \frac{1}{a_{r+1}} \bigl( C_2 + a_2 E_2 + \ldots + a_{r+1} E_{r+1} \bigr) \\
					&\equiv \Bigl( \frac{1}{2} - \frac{1}{a_{r+1}} \Bigr) C_0.
				\end{align*}
				It remains to prove $a_{n+1} = n$. This directly follows from Proposition~\ref{prop:SingIndexAndFiberCoefficients}. It is immediate to see that $C_0^2 = 0$, $C_0$ is nef, and $C_0$ is not numerically trivial. The same properties are true for $\pi^*K_X$. \qedhere
		\end{enumerate}
	\end{proof}
	\begin{proposition}\label{prop:SingIndexAndFiberCoefficients}
		Suppose that $C_2 \cup E_2 \cup \ldots \cup E_r$ has the configuration\!\!%
	\raisebox{-11pt}[0pt][13pt]{
		\begin{tikzpicture}
			\draw(0,0) node[anchor=center] (E1) {};
			\draw(30pt,0pt) node[anchor=center] (E2) {};
			\draw(60pt,0pt) node[anchor=center, inner sep=10pt] (E3) {};
			\draw(90pt,0pt) node[anchor=center] (E4) {};
			\fill[black] (E1) circle (1.5pt);
			\fill[black] (E2) circle (1.5pt);
			\draw (E3)	node[anchor=center]{$\cdots$};
			\fill[black] (E4) circle (1.5pt);
			\node[black,below,shift=(90:2pt)] at (E1.south) {$\scriptscriptstyle -k_1$};
			\node[black,below,shift=(90:2pt)] at (E2.south) {$\scriptscriptstyle -k_2$};
			\node[black,below,shift=(90:2pt)] at (E4.south) {$\scriptscriptstyle -k_r$};
			\draw[-] (E1.east) -- (E2.west);
			\draw[-] (E2.east) -- (E3.west);
			\draw[-] (E3.east) -- (E4.west);
		\end{tikzpicture} }\!\!, %
		so that it contracts to give a $T_1$-singularity of type $\bigl( 0 \in \A^2 \big/ \frac{1}{n^2}(1,na-1)\bigr)$. Then, in the expression
		\[
			C_2 + a_2 E_2 + \ldots + a_{r+1} E_{r+1}
		\]
		of the fiber (\ref{eq:SpecialFiber}), the coefficient of the $(-k_1)$-curve is $a$, and the coefficient of the $(-k_r)$-curve is $(n-a)$. Furthermore, $a_{r+1}$ equals to the sum of these two coefficients, hence  $a_{r+1} = n$.
	\end{proposition}
	\begin{proof}
		The proof proceeds by an induction on $r$. The case $r = 2$ is trivial. Indeed, a simple computations shows that $n = 3$, $a = 1$, and $a_2 = 2$, $a_3=  3$. To make notations simpler, we reindex $\{C_2,\, E_2,\,\ldots,\, E_{r+1}\}$ as follows: 
		\[
			(G_1,\,G_2,\,\ldots,\,G_{r+1}) = (E_{i_k},\,E_{i_{k-1}},\,\ldots,\,E_{i_1},\,C_2,\,E_{j_1},\,\ldots,\,E_{j_\ell},\,E_{r+1}).\hskip-35pt\tag{Figure~\ref{fig:Configuration_General}}
		\]
		By the induction hypothesis, we may assume 
		\[
			C_2 + a_2E_2 + \ldots + a_{r+1} E_{r+1} = a G_1 + \ldots + (n-a) G_r + n G_{r+1}.
		\]
		Let $\varphi_1 \colon \widetilde Y \to Y$ be the blow up at the point $G_{r+1} \cap G_1$, let $\widetilde G_i$\,($i=1,\ldots,r+1$) be the proper transform of $G_i$, and let $\widetilde G_{r+2}$ be the exceptional divisor. The $(-1)$-curve $\widetilde G_{r+2}$ meets $\widetilde G_1$ and $\widetilde G_{r+1}$ transversally, so
		\begin{align*}
			\varphi^*( aG_1 + \ldots + nG_{r+1}) 
			&= a ( \widetilde G_1 + \widetilde G_{r+2}) + g_2 \widetilde G_2 + \ldots + (n-a) \widetilde G_r + n( \widetilde G_{r+1} + \widetilde G_{r+2}) \\
			&= a \widetilde G_1 + g_2\widetilde G_2 + \ldots + (n-a) \widetilde G_r + n\widetilde G_{r+1} + (n+a) \widetilde G_{r+2}.
		\end{align*}
		It is well-known that the contraction of $\widetilde G_1, \ldots, \widetilde G_{r+1} \subset \widetilde Y$ produces a cyclic quotient singularity of type 
		\[
			\Bigl( 0 \in \A^2 \Big/ \frac{1}{(n+a)^2}(1,(n+a)n-1) \Bigr).
		\]
		This proves the statement for the chain $\widetilde G_1 \cup \ldots \cup \widetilde G_{r+2}$, so we are done by the induction. The same argument also works if one performs the blow up $\varphi_2 \colon \widetilde Y' \to Y$ at the point $G_{r+1} \cap G_r$.
	\end{proof}
	Now we want to obtain a smooth surface via a $\Q$-Gorenstein smoothing of $X$. It is well-known that $T_1$-singularities admit local $\Q$-Gorenstein smoothings, thus we have to verify that:
	\begin{enumerate}
		\item every formal deformation of $X$ is algebraizable;
		\item every local deformation of $X$ can be globalized.
	\end{enumerate}
	The answer for (a) is an immediate consequence of Grothendieck's existence theorem\,\cite[Example~21.2.5]{Hartshorne:DeformationTheory} since $H^2(\mathcal O_X)=0$. The next lemma verifies (b).
	\begin{lemma}\label{lem:NoObstruction}
		Let $Y$ be the nonsingular rational elliptic surface introduced above, and let $\mathcal T_Y$ be the tangent sheaf of $Y$. Then, 
		\[
			H^2(Y, \mathcal T_Y( - C_1 - C_2 - E_2 - \ldots - E_r ) ) = 0.
		\]
		In particular, $H^2(X,\mathcal T_X) = 0$\,(see \cite[Theorem~2]{LeePark:SimplyConnected}).
	\end{lemma}
	\begin{proof}
		The proof is not very different from \cite[\textsection4, Example~2]{LeePark:SimplyConnected}. The main claim is 
		\[
			H^0(Y, \Omega_Y^1(K_Y + C_1 + C_2 + E_2 + \ldots + E_r)) =0.
		\]
		By Lemma~\ref{lem:CanonicalofY} and equation~(\ref{eq:SpecialFiber}), 
		\[
			K_Y + C_1 + C_2 + E_2 + \ldots + E_r = C_0 - E_1 - E_{r+1}.
		\]
		Then, $h^0(Y,\Omega_Y^1(C_0 - E_1 - E_{r+1})) \leq h^0(Y,\Omega_Y^1(C_0)) = h^0(Y',\Omega_{Y'}^1(C_0'))$ where $Y'= \op{Bl}_9\P^2$, and $h^0(Y',\Omega_{Y'}^1(C_0')) =0$ by \cite[\textsection4, Lemma~2]{LeePark:SimplyConnected}. The result directly follows from the Serre duality.
	\end{proof}
		We have shown that the surface $X$ admits a $\Q$-Gorenstein smoothing $\mathcal X \to T$. The next aim is to show that the general fiber $X^\gen := \mathcal X_t$ is a Dolgachev surface of type $(2,n)$.
		\begin{proposition}\label{prop:CohomologyComparison_YtoX}
			Let $X$ be a projective normal surface with only rational singularities, let $\pi \colon Y \to X$ be a resolution of singularities, and let $E_1,\ldots,E_r$ be the exceptional divisors. If $D$ is a divisor on $Y$ such that $(D.E_i)=0$ for all $i=1,\ldots,r$, then
			\[
				H^p(Y,D) \simeq H^p(X,\pi_*D)
			\]
			for all $p \geq 0$.
		\end{proposition}
		\begin{proof}
			Since the singularities of $X$ are rational, each $E_i$ is a smooth rational curve. The assumption on $D$ in the statement implies that $\pi_*D$ is Cartier\,\cite[Theorem~12.1]{Lipman:RationalSings}, and $\pi^*\mathcal O_X(\pi_*D) = \mathcal O_Y(D)$. By the projection formula, $R^p\pi_*\mathcal O_Y(D) \simeq R^p \pi_*( \mathcal O_Y \otimes \pi^* \mathcal O_X(\pi_*D) ) \simeq (R^p \pi_* \mathcal O_Y) \otimes \mathcal O_X(\pi_*D)$. Since $X$ is normal and has only rational singularities,
			\[
				R^p\pi_* \mathcal O_Y = \left\{
				\begin{array}{ll}
					\mathcal O_X & \text{if } p=0\\
					0 & \text{if } p > 0.
				\end{array}
				\right.
			\]
			Now, the claim is an immediate consequence of the Leray spectral sequence
			\[
				E_2^{p,q} = H^p(X,\, R^q\pi_*\mathcal O_Y \otimes \mathcal O_X(\pi_*D) )  \Rightarrow H^{p+q}(Y,\, \mathcal O_Y(D)). \qedhere
			\]
		\end{proof}
		\begin{lemma}\label{lem:Cohomologies_ofGeneralFiber_inY}
			Let $\pi \colon Y \to X$ be the contraction defined in Proposition~\ref{prop:SingularSurfaceX}. Then,
			\[
				h^0(X,\pi_*C_0) = 2,\quad h^1(X,\pi_*C_0)=1,\ \text{and}\quad h^2(X,\pi_*C_0)=0.
			\]
		\end{lemma}
		\begin{proof}
			It is easy to see that $(C_0.C_1) = (C_0.C_2) = (C_0.E_2) =\ldots = (C_0.E_r) = 0$. Hence by Proposition~\ref{prop:CohomologyComparison_YtoX}, it suffices to compute $h^p(Y,C_0)$. Since $C_0^2 = (K_Y . C_0)=0$, the Riemann-Roch formula shows $\chi(C_0)=1$. By Serre duality, $h^2(C_0) = h^0(K_Y - C_0)$. In the short exact sequence
			\[
				0 \to \mathcal O_Y(K_Y - C_0 - E_1) \to \mathcal O_Y(K_Y -C_0) \to \mathcal O_{E_1} \otimes \mathcal O_Y(K_Y - C_0)\to 0,
			\]
			we find that $H^0(\mathcal O_{E_1} \otimes \mathcal O_Y(K_Y - C)) = 0$ since $(K_Y - C_0 \mathbin . E_1) = -1$. It follows that
			\[
				h^0(K_Y-C_0) = h^0(K_Y-C_0-E_1),
			\]
			but $K_Y - C_0 - E_1 = - C_0 - C_2 - E_2 - \ldots - E_{r+1}$ by Lemma~\ref{lem:CanonicalofY}. Hence $h^2(C_0)=0$. Since the complete linear system $\lvert C_0 \rvert$ defines the elliptic fibration $Y \to \P^1$, $h^0(C_0) = 2$. Furthermore, $h^1(C_0)=1$ follows from $h^0(C_0)=2$, $h^2(C_0)=0$, and $\chi(C_0)=1$.
		\end{proof}
		The following proposition, due to Manetti\,\cite{Manetti:NormalDegenerationOfPlane}, is a key ingredient of the proof of Theorem~\ref{thm:SmoothingX}
		\begin{proposition}[{\cite[Lemma~2]{Manetti:NormalDegenerationOfPlane}}]\label{prop:Manetti_PicLemma}
			Let $\mathcal X \to ( 0 \in T)$ be a smoothing of a normal surface $X$ with $H^1(\mathcal O_X) = H^2(\mathcal O_X)=0$. Then for every $t \in T$, the natural restriction map of the second cohomology groups $H^2(\mathcal X,\Z) \to H^2(\mathcal X_t,\Z)$ induces an injection $\Pic \mathcal X \to \Pic \mathcal X_t$. Furthermore, the restriction to the central fiber $\Pic \mathcal X \to \Pic X$ is an isomorphism.
		\end{proposition}
		\begin{theorem}\label{thm:SmoothingX}
			Let $X$ be the projective normal surface defined in Proposition~\ref{prop:SingularSurfaceX}, and let $\varphi \colon \mathcal X \to (0 \in T)$ be a one parameter $\Q$-Gorenstein smoothing of $X$ over a smooth curve germ $(0 \in T)$. For a general point $0 \neq t_0 \in T$, the fiber $X^\gen := \mathcal X_{t_0}$ satisfies the following:
			\begin{enumerate}
				\item $p_g(X^\gen) = q(X^\gen) = 0$;
				\item $X^\gen$ is a simply connected, minimal, nonsingular surface with Kodaira dimension $1$;
				\item there exists an elliptic fibration $f^\gen \colon X^\gen \to \P^1$ such that $K_{X^\gen} \equiv C_0^\gen - \frac{1}{2} C_0^\gen - \frac{1}{n} C_0^\gen$, where $C_0^\gen$ is the general fiber of $f^\gen$. In particular, $X^\gen$ is isomorphic to the Dolgachev surface of type $(2,n)$.
			\end{enumerate}
		\end{theorem}
		\begin{proof}\ 
			\begin{enumerate}[ref=(\normalfont\alph{enumi})]
				\item This follows from Proposition~\ref{prop:SingularSurfaceX}\ref{item:SingularSurfaceX_Cohomologies} and the upper-semicontinuity of $h^p$.
				\item Shrinking $(0 \in T)$ if necessary, we may assume that $X^\gen$ is simply connected\,\cite[p.~499]{LeePark:SimplyConnected}, and that $K_{X^\gen}$ is nef\,\cite[\textsection5.d]{Nakayama:ZariskiDecomposition}. If $K_{X^\gen}$ is numerically trivial, then $X^\gen$ must be an Enriques surface by the classification theory of surfaces. This violates the simple connectivity of $X^\gen$. It follows that $K_{X^\gen}$ is not numerically trivial, and the Kodaira dimension of $X^\gen$ is $1$.
				\item %

				Since the divisor $\pi_*C_0$ is not supported on the singular points of $X$, $\pi_* C_0 \in \Pic X$. By Proposition~\ref{prop:Manetti_PicLemma}, $\Pic X \simeq \Pic \mathcal X \hookrightarrow \Pic X^\gen$. Let $C_0^\gen\in \Pic X^\gen$ be the image of $\pi_*C_0$ under this correspondence. By \cite[Theorem~4.2]{Kawamata:Moderate}, there exists a smooth complex surface $B$ such that the morphism $\varphi$ factors through $g \colon \mathcal X \to B$ and the general fiber of $g$ is an elliptic curves. In particular, the complete linear system $\lvert C_0^\gen\rvert$ defines the elliptic fibration $f^\gen \colon X^\gen \to \P^1$.
				Since $\mathcal X / (0 \in T)$ is a $\Q$-Gorenstein deformation, the map $\Pic X \hookrightarrow \Pic X^\gen$ in Proposition~\ref{prop:Manetti_PicLemma} maps $2nK_X - (n-2) \pi_* C_0$ to $2nK_{X^\gen} - (n-2) C_0^\gen$. Furthermore, $2nK_X - (n-2)\pi_*C_0 \in \Pic X$ is zero, so
				\[
					K_{X^\gen} \equiv C_0^\gen - \frac{1}{2}C_0^\gen - \frac{1}{n}C_0^\gen.
				\]
				
				By \cite[Chapter~2]{Dolgachev:AlgebraicSurfaces}, every minimal simply connected nonsingular surface with $p_g=q=0$ and of Kodaira dimension $1$ has exactly two multiple fibers with coprime multiplicities. Thus, there exist coprime integers $q > p > 0$ such that $X^\gen \simeq X_{p,q}$ where $X_{p,q}$ is a Dolgachev surface of type $(p,q)$. The canonical bundle formula says that $K_{X_{p,q}} \equiv C_0^\gen - \frac 1p C_0^\gen - \frac 1q C_0^\gen$. Since $X^\gen \simeq X_{p,q}$, this leads to the equality
				\[
					\frac 12 + \frac 1n = \frac 1p + \frac 1q.
				\]
				Assume $2 < p < q$. Then, $\frac 12 < \frac 12 + \frac 1n  = \frac 1p + \frac 1q \leq \frac 13 + \frac 1q$. Hence, $q < 6$. Only the possible candidates are $(p,q,n) = (3,4,12)$, $(3,5,30)$, but all of these cases violate $\op{gcd}(2,n) = 1$. It follows that $p=2$ and $q = n$.\qedhere
			\end{enumerate}
		\end{proof}
		\begin{remark}
			Theorem~\ref{thm:SmoothingX} generalizes to the construction of Dolgachev surfaces of type $(m,n)$ for any coprime integers $n>m>0$. Indeed, we shall describe the multiple fiber of multiplicity $n$ associated to the Weil divisor $\pi_* E_{r+1}$. The precise meaning of this sentence will be explained in the next section\,(see Example~\ref{eg:MultipleFibers}). If we perform more blow ups to the $C_1\cup E_1$-fiber so that $X$ contains a $T_1$-singularity of type $\bigl( 0 \in \A^2 \big / \frac{1}{m^2}(1,mb-1)\bigr)$, then the surface $X^\gen$ has two multiple fibers of multiplicities $m$ and $n$. Thus, $X^\gen$ is a Dolgachev surface of type $(m,n)$.
		\end{remark}
	\section{Exceptional vector bundles on Dolgachev surfaces}\label{sec:ExcepBundleOnX^g}
		In general, it is hard to understand how information of the central fiber is carried to the general fiber along the $\Q$-Gorenstein smoothing. Looking at the topology near the singularities of $X$, one can get a clue to relate information between $X$ and $X^\gen$. This section essentially follows the idea of Hacking. Some ingredients of Hacking's method, which are necessary for our application, are included in the appendix(Section~\ref{sec:Appendix}). Readers who want to look up the details are recommended to consult Hacking's original paper\,\cite{Hacking:ExceptionalVectorBundle}. 
	\subsection{Topology of the singularities of $X$}\label{subsec:TopologyofX}
		Let $L_i \subseteq X$\,($i=1,2$) be the link of the singularity $P_i$. Then, $H_1(L_1,\Z) \simeq \Z/4\Z$ and $H_1(L_2,\Z) \simeq \Z/n^2\Z$\,({\it cf.} \cite[Proposition~13]{Manetti:NormalDegenerationOfPlane}). Since $\op{gcd}(2,n)=1$, $H_1(L_1,\Z) \oplus H_1(L_2,\Z) \simeq \Z/ 4n^2 \Z$ is a finite cyclic group. By \cite[p.~1191]{Hacking:ExceptionalVectorBundle}, $H_2(X,\Z) \to H_1(L_i,\Z)$ is surjective for each $i=1,2$, thus the natural map 
		\[
			H_2(X,\Z) \to H_1(L_1,\Z) \oplus H_1(L_2,\Z),\quad \alpha \mapsto ( \alpha \cap L_1 ,\, \alpha \cap L_2)
		\]
		is surjective. We have further information on groups the $H_1(L_i,\Z)$.
		\begin{theorem}[{\cite{Mumford:TopologyOfNormalSurfaceSingularity}}]\label{thm:MumfordTopologyOfLink}
			Let $X$ be a projective normal surface containing a $T_1$-singularity $P \in X$. Let $f \colon \widetilde X \to X$ be a good resolution ({\it i.e.} the exceptional divisor is simple normal crossing) of the singularity $P$, and let $E_1,\ldots,E_r$ be the integral exceptional divisors ordered in such a way that $(E_i . E_{i+1}) = 1$ for each $i=1,\ldots,r-1$. Let $\widetilde L\subseteq \widetilde X$ be the plumbing fixture (see Figure~\ref{fig:PlumbingFixture}) around $\bigcup E_i$, and let $\alpha_i \subset \widetilde L$ be the loop around $E_i$ oriented suitably. Then the following statements are true.
			\begin{enumerate}
				\item The group $H_1(\widetilde L,\Z)$ is generated by the loops $\alpha_i$. The relations are
				\[
					\sum_j (E_i . E_j) \alpha_j = 0,\quad i=1,\,\ldots,\,r.
				\]
				\item Let $L \subset X$ be the link of the singularity $P \in X$. Then, $\widetilde L$ is homeomorphic to $L$.
			\end{enumerate}
		\begin{figure}[h!]
			\centering
			\begin{tikzpicture}[scale=1]
				\draw(0,0)	node[anchor=center] {\includegraphics[scale=1, bb = -15 -10 268 98]{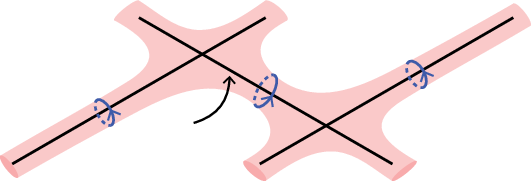}};
				\draw(0.5,0.5)	node[anchor=center] {\color{blue}{$\alpha_i$}};
				\draw(-1.4,-0.5) node[anchor=center] {$E_i$};
			\end{tikzpicture}\vskip-15pt
			\caption{Plumbing fixture around $\bigcup E_i$.}\label{fig:PlumbingFixture}
		\end{figure}
		\end{theorem}
		Proposition~\ref{prop:Manetti_PicLemma} provides a way to associate a Cartier divisor on $X$ with a Cartier divisor on $X^\gen$.  This association can be extended as the following proposition illustrates.
		\begin{proposition}[{{\it cf.} \cite[Lemma~5.5]{Hacking:ExceptionalVectorBundle}}]\label{prop:Hacking_Specialization}
			Let $X$ be a projective normal surface, and let $(P \in X)$ be a $T_1$-singularity of type $\bigl( 0 \in \A^2 \big/ \frac{1}{n^2}(1,na-1)\bigr)$. Suppose $X$ admits a $\Q$-Gorenstein deformation $\mathcal X/(0 \in T)$ over a smooth curve germ $(0 \in T)$ such that $\mathcal X / (0 \in T)$ is a smoothing of $(P \in X)$, and is locally trivial outside $(P \in X)$. Let $X^\gen$ be the general fiber of $\mathcal X \to (0 \in T)$, and let $\mathcal B \subset \mathcal X$ be a sufficiently small open ball around $P \in \mathcal X$. Then the link $L$ and the Milnor fiber $M$ of $(P \in X)$ given as follows:
			\[
				L = \partial \mathcal B \cap X^\gen,\qquad M = \mathcal B \cap X^\gen.
			\]
			In addition, let $B := \mathcal B \cap X$ be the contractible space. Then, the relative homology sequence for the pair $(X^\gen, M)$ yields the exact sequence
			\[
				0 \to H_2(X^\gen,\Z) \to H_2(X,\Z) \to H_1(M,\Z).
			\]
			Furthermore, a class in $H_2(X,\Z)$ lifts to a class in $H_2(X^\gen,\Z)$ if and only if its image under the map $H_2(X,\Z) \to H_1(L,\Z)$ is divisible by $n$.
        \end{proposition}
        \begin{proof}
			We have a sequence of isomorphisms
			\[
				H_2(X^\gen,M) \simeq H_2(X^\gen \setminus M , \partial M) \simeq H_2(X\setminus B , \partial B) \simeq H_2(X,B) \simeq H_2(X).
			\]
			where the first and the third ones are the excisions, the second one is due to the topological description $X^\gen = (X \setminus B) \cup M$\,(\cite[p.~39]{Manetti:ModuliOfDiffeo}), and the last one is due to the contractibility of $B$. The relative homology sequence for the pair $(X^\gen, M)$ gives
			\[
				0 \to H_2(X^\gen) \to H_2(X^\gen,M) \simeq H_2(X) \to H_1(M).
			\]
			The map in the right is the composition $H_2(X) \to H_1(L) \to H_1(M)$, where $H_1(L) \to H_1(M)$ is the natural surjection $\Z/n^2 \Z \to \Z/n\Z$\,(\cite[Lemma~2.1]{Hacking:ExceptionalVectorBundle}). The last assertion follows immediately.
        \end{proof}
       
		Recall that $Y$ is the rational elliptic surface constructed in Section~\ref{sec:Construction}, and $\pi \colon Y \to X$ is the contraction of $C_1,\,C_2,\,E_2,\,\ldots,\,E_r$. Proposition~\ref{prop:Hacking_Specialization} gives the short exact sequence
		\begin{equation}
			0 \to H_2(X^\gen,\Z) \to H_2(X,\Z) \to H_1(M_1,\Z) \oplus H_1(M_2,\Z), \label{eq:SpecializationSequence}
		\end{equation}
		where $M_1$\,(resp. $M_2$) is the Milnor fiber of the smoothing of $(P_1 \in X)$\,(resp. $(P_2 \in X)$). In this case $H_2(X,\Z) \to H_1(L_1,\Z)\oplus H_1(L_2,\Z)$ is described as follows. If $D \in \Pic Y$, then $[\pi_*D] \in H_2(X,\Z)$ maps to
		\[
			\bigl( (D.C_1) \alpha_{C_1},\ (D.C_2) \alpha_{C_2} + (D.E_2) \alpha_{E_2} + \ldots + (D.E_r) \alpha_{E_r}\bigr).
		\]
		Suppose $D \in \Pic Y$ is a divisor such that $(D. C_1) \in 2\Z$, $(D. C_2) \in n\Z$, and $(D.E_2) = \ldots = (D.E_r) = 0$. Then, Theorem~\ref{thm:MumfordTopologyOfLink} and (\ref{eq:SpecializationSequence}) imply that the cycle $[\pi_* D] \in H_2(X,\Z)$ maps to the zero element of $H_1(M_1,\Z) \oplus H_1(M_2,\Z)$. In particular, there is a cycle in $H_2(X^\gen)$, which maps to $[\pi_* D]$. Since $X^\gen$ is a nonsingular surface with $p_g = q = 0$, the first Chern class map and Poincar\'e duality induce the isomorphisms $\Pic X^\gen \simeq H^2(X^\gen,\Z) \simeq H_2(X^\gen,\Z)$\,(\cite[Proposition~4.11]{Kollar:Seifert}). We take the divisor $D^\gen \in \Pic X^\gen$ corresponding to $[\pi_* D] \in H_2(X,\Z)$. More detailed description of $D^\gen$ will be presented in \textsection\ref{subsec:HackingConstr}. We remark that even if $\pi_*D$ is an effective divisor, it does not necessarily mean that the resulting divisor $D^\gen$ is effective.
		\begin{example}\label{eg:MultipleFibers}
			If $D = E_{r+1}$, then $[\pi_*E_{r+1}] \in H_2(X,\Z)$ maps to $(0, \alpha_{E_r} + \alpha_{E_s} ) \in H_1(L_1,\Z) \oplus H_1(L_2,\Z)$, where $E_s$ is the other end component of the chain $\{C_2,E_2,\ldots,E_r\}$. It can be easily shown that either $\alpha_{E_s} = (na-1) \alpha_{E_r}$ or $\alpha_{E_r} = (na-1) \alpha_{E_s}$. In both cases, $\alpha_{E_r} + \alpha_{E_s}$ maps to the zero cycle along $H_1(L_2,\Z) \to H_1(M_2,\Z)$. It follows that $[\pi_*E_{r+1}] \in H_2(X,\Z)$ admits a preimage $E_{r+1}^\gen$ in $H_2(X^\gen,\Z) \simeq \Pic X^\gen$. By (\ref{eq:EquationOnFiber}) and Proposition~\ref{prop:SingIndexAndFiberCoefficients}, there are integers $a_2,\ldots,a_r \in \Z$ such that
			\[
				C_0 = C_2 + a_2 E_2 + \ldots + a_r E_r + n E_{r+1}.
			\]
			This leads to $\pi_* C_0 = \pi_* ( C_2 + a_2 E_2 + \ldots + a_r E_r + n E_{r+1}) = \pi_* ( n E_{r+1})$. Since $\mathcal X/(0 \in T)$ is a $\Q$-Gorenstein deformation, $\pi_*( (n-2) E_{r+1} ) \equiv \frac{n-2}{n}\pi_* C_0 \equiv 2K_X$\,(Proposition~\ref{prop:SingularSurfaceX}) induces $(n-2)E_{r+1}^\gen = 2K_{X^\gen}$. The same argument says that there exists $E_1^\gen \in \Pic X^\gen$ with $(n-2)E_1^\gen = nK_{X^\gen}$. In particular, we find that both $2E_1^\gen$ and $nE_{r+1}^\gen$ are $\Q$-linearly equivalent to $\frac{2n}{n-2}K_{X^\gen}$, which is again $\Q$-linearly equivalent to the general fiber $C_0^\gen$.
		\end{example}
		
	The next proposition explains the way to find a preimage along the surjective map $H_2(X,\Z) \to H_1(L_1,\Z)\oplus H_1(L_2,\Z)$.
	\begin{proposition}\label{prop:DesiredDivisorOnY}
		Let $X$ be a projective normal surface with a cyclic quotient singularity $(P \in X)$, let $\pi \colon Y \to X$ be a resolution of $P \in X$, and let $E_1,\ldots,E_r \subset Y$ be the exceptional divisors over $P$. The first homology group of the link $L$ has the following presentation
		\[
			\bigl\langle \alpha_1,\ldots,\alpha_r : \sum_{j=1}^r (E_i . E_j) \alpha_j = 0,\ i=1,\ldots,r \bigr\rangle.
		\]
		Let $D$ be a divisor on $Y$, and let $\ell_1,\ldots,\ell_r$ be integers satisfying
		\[
			[\pi_* D] \cap L = \ell_1 \alpha_1 + \ldots + \ell_r \alpha_r
		\]
		in $H_1(L)$. Then there are integers $e_1,\ldots,e_r$ such that $D' := D + \sum_{j=1}^r e_jE_j $ satisfies $(D'.E_i) = \ell_i$ for each $i=1,\ldots,r$.
	\end{proposition}
	\begin{proof}
		Consider the free abelian group $\bigoplus_{j=1}^r \Z \cdot \tilde{\alpha}_j$ and the homomorphism
		\[
			\bigoplus_{j=1}^r \Z \cdot \tilde\alpha_j \to H_1(L),\quad \tilde\alpha_j \mapsto \alpha_j.
		\]
		This map is clearly surjective. By Theorem~\ref{thm:MumfordTopologyOfLink}, the kernel is the abelian group generated by
		\[
			\Bigl\{ R_i := \sum_{j=1}^r(E_i . E_j) \alpha_j : i=1,\ldots,r\Bigr\}.
		\]
		Since $[\pi_*D] \cap L = \sum_{i=1}^r ( D. E_j) \alpha_j$, the equality $[\pi_*D] \cap L = \sum_{j=1}^r \ell_j \alpha_j$ implies that there are integers $e_1,\ldots,e_r$ such that $\sum_{j=1}^r \ell_j \tilde\alpha_j - \sum_{j=1}^r ( D.E_j) \tilde \alpha_j = e_1 R_1 + \ldots + e_r R_r$. This leads to
		\begin{align*}
			\sum_{j=1}^r \ell_j \tilde \alpha_j &= \sum_{j=1}^r(D.E_j) \tilde \alpha_j + \sum_{i=1}^r \sum_{j=1}^r (e_i E_i . E_j) \tilde \alpha_j \\
			&= \sum_{j=1}^r ( D + e_1 E_1 + \ldots + e_r E_r \mathbin. G_j) \tilde\alpha_j.
		\end{align*}
		Taking $D' = D + e_1 E_1 + \ldots + E_r$, we get $(D'.E_i)=\ell_i$ for each $i=1,\ldots,r$.
	\end{proof}
	\subsection{Exceptional vector bundles on $X^\gen$}\label{subsec:HackingConstr}
		We keep the notations in Section~\ref{sec:Construction}, namely, $Y$ is the rational elliptic surface\,(Figure~\ref{fig:Configuration_General}), $\pi \colon Y \to X$ is the contraction in Proposition~\ref{prop:SingularSurfaceX}.
		Let $(0 \in T)$ be the base space of the versal deformation ${\mathcal X^{\rm ver} / (0 \in T)}$ of $X$, and let $(0 \in T_i)$ be the base space of the versal deformation $(P_i \in \mathcal X^{\rm ver}) / (0 \in T_i)$ of the singularity $(P_i \in X)$. By Lemma~\ref{lem:NoObstruction} and \cite[Lemma~7.2]{Hacking:ExceptionalVectorBundle}, there exists a smooth morphism
		\[
			\mathfrak T \colon (0 \in T) \to \textstyle\prod_i (0 \in T_i).
		\]
		For each $i=1,2$, take the base extensions $( 0 \in T_i') \to (0 \in T_i)$ to which Proposition~\ref{prop:HackingWtdBlup} can be applied. Then, there exists a Cartesian diagram
		\[
			\begin{xy}
				(0,0)*+{(0 \in T)}="00";
				(30,0)*+{\textstyle \prod_i ( 0 \in T_i)}="10";
				(0,15)*+{(0 \in T')}="01";
				"10"+"01"-"00"*+{\textstyle \prod_i ( 0 \in T_i')}="11";
				{\ar^(0.44){\mathfrak T}	"00";"10"};
				{\ar^(0.44){\mathfrak T'}	"01";"11"};
				{\ar		"01";"00"};
				{\ar		"11";"10"};
			\end{xy}.
		\]
		Let $\mathcal X' / (0 \in T')$ be the deformation obtained by pulling back $\mathcal X^{\rm ver} / ( 0 \in T)$ along $(0 \in T') \to (0 \in T)$. By Proposition~\ref{prop:HackingWtdBlup}, there exists a proper birational map $\Phi \colon \tilde{\mathcal X} \to \mathcal X'$ such that the central fiber $\tilde {\mathcal X}_0 = \Phi^{-1}(\mathcal X_0')$ is the union of three irreducible components $\tilde X_0$, $W_1$, $W_2$, where $\tilde X_0$ is the proper transform of $X = \mathcal X_0'$, and $W_1$\,(resp. $W_2$) is the exceptional locus over $P_1$\,(resp. $P_2$). The intersection $Z_i := \tilde X_0 \cap W_i$\,($i=1,2$) is a smooth rational curve.
				
		From now on, assume $a=1$. This is the case in which the resolution graph of the singular point $P_2 \in X$ forms the chain $C_2,\,E_2,\,\ldots,\,E_r$ in this order. Indeed, the resolution graph of a cyclic quotient singularity $\bigl(0 \in \A^2 / \frac{1}{n^2}(1,n-1)\bigr)$ is \!\!\!\!\!%
		\raisebox{-11pt}[0pt][13pt]{
			\begin{tikzpicture}
				\draw(0,0) node[anchor=center] (E1) {};
				\draw(30pt,0pt) node[anchor=center] (E2) {};
				\draw(60pt,0pt) node[anchor=center, inner sep=10pt] (E3) {};
				\draw(90pt,0pt) node[anchor=center] (E4) {};
				\fill[black] (E1) circle (1.5pt);
				\fill[black] (E2) circle (1.5pt);
				\draw (E3)	node[anchor=center]{$\cdots$};
				\fill[black] (E4) circle (1.5pt);
				\node[black,below,shift=(90:2pt)] at (E1.south) {$\scriptscriptstyle -(n+2)\ $};
				\node[black,below,shift=(90:2pt)] at (E2.south) {$\scriptscriptstyle -2$};
				\node[black,below,shift=(90:2pt)] at (E4.south) {$\scriptscriptstyle -2$};
				\draw[-] (E1.east) -- (E2.west);
				\draw[-] (E2.east) -- (E3.west);
				\draw[-] (E3.east) -- (E4.west);
			\end{tikzpicture} }\!\!($(n-1)$ vertices). %
		Let $\iota \colon Y \to \tilde X_0$ be the contraction of $E_2,\ldots,E_r$\,(see Proposition~\ref{prop:HackingWtdBlup}\ref{item:prop:HackingWtdBlup}). As noted in Remark~\ref{rmk:SimplestSingularCase}, $W_1$ is isomorphic to $\P^2$, $Z_1$ is a smooth conic in $W_1$, hence $\mathcal O_{W_1}(1)\big\vert_{Z_1} = \mathcal O_{Z_1}(2)$. Also,
		\begin{equation}
			W_2 \simeq \P_{x,y,z}(1,n-1,1),\quad Z_2 = (xy=z^n) \subset W_2,\quad \text{and}\quad \mathcal O_{W_2}(n-1)\big\vert_{Z_2} = \mathcal O_{Z_2}(n). \label{eq:SecondWtdBlowupExceptional}
		\end{equation}
		The last equality can be verified as follows: let $h_{W_2} = c_1(\mathcal O_{W_2}(1))$, then $(n-1)h_{W_2}^2 = 1$, so \linebreak$\bigl( c_1(\mathcal O_{W_2}(n-1)) \mathbin . Z_2\bigr) = \bigl( (n-1)h_{W_2} \mathbin. nh_{W_2} \bigr) = n$.

		In what follows, we construct exceptional vector bundles on the reducible surface $\tilde{\mathcal X_0} = \tilde X_0 \cup W_1 \cup W_2$ by gluing suitable vector bundles on each irreducible component which have isomorphic restrictions to the intersection curves $Z_i$. 
		\begin{proposition}\label{prop:VectorBundleOnReducibleSurface}
			Let $D \in \Pic Y$ be a divisor such that $(D.E_i) = 0$ for $i=2,\ldots,r$.
			\begin{enumerate}
				\item Assume that $(D.C_1) =2d_1 \in 2\Z$, $(D.C_2) = nd_2\in n\Z$. Then, there exists a line bundle $\tilde {\mathcal D}$ on the reducible surface $\tilde{\mathcal X}_0 = \tilde X_0 \cup W_1 \cup W_2$ satisfying
			\[
				\tilde{\mathcal D}\big\vert_{\tilde X_0} \simeq \mathcal O_{\tilde X_0}(\iota_* D),\quad%
				\tilde{\mathcal D}\big\vert_{W_1} \simeq \mathcal O_{W_1}(d_1),\quad\text{and}\quad%
				\tilde{\mathcal D}\big\vert_{W_2} \simeq \mathcal O_{W_2}((n-1)d_2).
			\]
				\item Assume that $(D.C_1) = 1$, $(D.C_2) = 0$, and that there exists an exceptional vector bundle $G_1$ of rank $2$ on $W_1$ such that $G_1\big\vert_{Z_1} \simeq \mathcal O_{Z_1}(1)^{\oplus 2}$. Then, there exists a vector bundle $\tilde{\mathcal V}_1$ on $\tilde{\mathcal X}_0$ satisfying
				\[
					\tilde{\mathcal V}_1 \big\vert_{\tilde X_0} \simeq \mathcal O_{\tilde X_0}(\iota_* D)^{\oplus 2},\quad%
					\tilde{\mathcal V}_1 \big\vert_{W_1} \simeq G_1,\quad\text{and}\quad%
					\tilde{\mathcal V}_1 \big\vert_{W_2} \simeq \mathcal O_{W_2}^{\oplus 2}.%
				\]
				\item Assume that $(D.C_1)=0$, $(D.C_2) = 1$, and that there exists an exceptional vector bundle $G_2$ of rank $n$ on $W_2$ such that $G_2 \big\vert_{Z_2} \simeq \mathcal O_{Z_2}(1)^{\oplus n}$. Then, there exists a vector bundle $\tilde{\mathcal V}_2$  on $\tilde{\mathcal X}_0$ satisfying
				\[
					\tilde{\mathcal V}_2 \big\vert_{\tilde X_0} \simeq \mathcal O_{\tilde X_0}(\iota_* D)^{\oplus n},\quad%
					\tilde{\mathcal V}_2 \big\vert_{W_1} \simeq \mathcal O_{W_1}^{\oplus n},\quad\text{and}\quad%
					\tilde{\mathcal V}_2 \big\vert_{W_2} \simeq G_2.%
				\]
			\end{enumerate}
			Furthermore, all the bundles introduced above are exceptional.
		\end{proposition}
		\begin{proof}
			For all of those three cases, the ``ingredient bundles'' on irreducible components have isomorphic restrictions on $Z_i$, hence $\tilde{\mathcal E}(= \tilde{\mathcal D},\,\tilde{\mathcal V}_1,\,\tilde{\mathcal V}_2)$ exists as a vector bundle in the exact sequence\,(\textit{cf.}~\cite[Lemma~7.3]{Hacking:ExceptionalVectorBundle})
		\begin{equation}
			0 \to \tilde{\mathcal E} \to \tilde{\mathcal E}\big\vert_{\tilde X_0} \oplus \tilde{\mathcal E}\big\vert_{W_1} \oplus \tilde{\mathcal E}\big\vert_{W_2} \to \tilde{\mathcal E}\big\vert_{Z_1} \oplus \tilde{\mathcal E} \big\vert_{Z_2} \to 0.\label{eq:ExactSeq_onReducibleSurface}
		\end{equation}
		Conversely, given any vector bundle on $\tilde{\mathcal X}_0$, one can consider the exact sequence of the form (\ref{eq:ExactSeq_onReducibleSurface}). We plug the corresponding endomorphism sheaf into the sequence (\ref{eq:ExactSeq_onReducibleSurface}) to verify that $\tilde{\mathcal E}$ is exceptional. 
		
		Replacing $\tilde{\mathcal E}$ by $\varEnd(\tilde{\mathcal D}) \simeq \tilde{\mathcal D}^\vee \otimes \tilde{\mathcal D} \simeq \mathcal O_{\tilde{\mathcal X}_0}$, we rewrite (\ref{eq:ExactSeq_onReducibleSurface}) as
			\[
				0 \to \mathcal O_{\tilde{\mathcal X}_0} \to \mathcal O_{\tilde X_0} \oplus \mathcal O_{W_1} \oplus \mathcal O_{W_2} \to \mathcal O_{Z_1} \oplus \mathcal O_{Z_2} \to 0.
			\]
			Looking at cohomologies, we can easily verify that $h^p(\mathcal O_{\tilde{\mathcal X}_0}) = h^p( \mathcal O_{\tilde X_0})$. Using the same argument as in Proposition~\ref{prop:SingularSurfaceX}(a), we find
			\[
				H^p(\mathcal O_{\tilde X_0}) = \left\{
				\begin{array}{ll}
					\C & p=0 \\
					0 & p\neq 0.
				\end{array}
				\right.
			\]
			Since $\tilde{\mathcal D}$ is locally free, $\varExt^q(\tilde{\mathcal D},\tilde{\mathcal D})=0$ for $q \neq 0$. By the local-to-global spectral sequence
			\[
				E_2^{p,q} = H^p( \varExt^q(\tilde{\mathcal D},\tilde{\mathcal D})) \Rightarrow H^{p+q}( \End(\tilde{\mathcal D})),
			\]
			$h^p(\End(\tilde{\mathcal D})) \simeq \dim_\C E_2^{p,0} = h^p(\mathcal O_{\tilde{\mathcal X}_0})$, showing that $\tilde{\mathcal D}$ is an exceptional line bundle. Now, we consider (\ref{eq:ExactSeq_onReducibleSurface}) for $\tilde{\mathcal E} = \varEnd(\tilde{\mathcal V}_1)$ which reads
			\[
				0 \to \varEnd(\tilde{\mathcal V}_1) \to \mathcal O_{\tilde X_0}^{\oplus 4} \oplus \varEnd(G_1) \oplus \mathcal O_{W_2}^{\oplus 4} \to \mathcal O_{Z_1}^{\oplus 4} \oplus \mathcal O_{Z_2}^{\oplus 4} \to 0.
			\]
			Since the restrictions $H^0(\mathcal O_{\tilde X_0}) \to H^0(\mathcal O_{Z_1})$, $H^0(\mathcal O_{W_2}) \to H^0(\mathcal O_{Z_2})$ are surjective, we have $h^p(\varEnd(\tilde{\mathcal V}_1)) = h^p(\varEnd(G_1))$. Using the local-to-global spectral sequences for the sheaves $\varEnd(\tilde{\mathcal V}_1)$, $\varEnd(G_1)$ and proceed as done in (a), we can conclude that $h^p(\End(\tilde{\mathcal V}_1)) = h^p(\varEnd(\tilde{\mathcal V}_1)) = h^p(\varEnd(G_1)) = h^p(\End(G_1))$, thus $\tilde{\mathcal V}_1$ is an exceptional vector bundle. Similarly, one can prove that $\tilde{\mathcal V}_2$ is an exceptional vector bundle. \qedhere
	\end{proof}
		
		We use Proposition~\ref{prop:DesiredDivisorOnY} to find the divisors satisfying the conditions described in Proposition~\ref{prop:VectorBundleOnReducibleSurface}(b) and (c).
		\begin{lemma}\label{lem: Divisor for Higher ranks}
			Let $N_1,N_2$ be solutions of the systems of congruence equations
			\[
				N_1 \equiv \left\{
					\begin{array}{ll}
						1 \mod 4 \\
						0 \mod n^2
					\end{array}
				\right. \qquad
				N_2 \equiv \left\{
					\begin{array}{ll}
						0 \mod 4 \\
						1 \mod n^2\,.
					\end{array}
				\right.
			\]
			Then, 
			\begin{enumerate}
				\item there are integers $e,e_1,\ldots,e_r \in \Z$ such that $V_1 := N_1 F_1 + e C_1 + e_1 C_2 + e_2 E_2 + \ldots + e_r E_r$ satisfies $(V_1.C_1)=1$ and $(V_1.C_2)=(V_1.E_2)=\ldots=(V_1.E_r)=0$;
				\item there are integers $f,f_1,\ldots,f_r \in \Z$ such that $V_2 := N_2 F_1 + f C_1 + f_1 C_2 + f_2 E_2 + \ldots + f_r E_r$ satisfies $(V_2.C_2)=1$ and $(V_2.C_1)=(V_2.E_2)=\ldots=(V_2.E_r)=0$.
			\end{enumerate}
		\end{lemma}
		\begin{proof}
			By the choices of $N_1,N_2$, we have
			\[
				\bigl( [\pi_* (N_i F_1)] \cap L_1 ,\, [\pi_*(N_i F_1)] \cap L_2 \bigr) = \left\{
					\begin{array}{ll}
						( \alpha_{C_1},\, 0 ) & i=1 \\
						( 0,\, \alpha_{C_2}) & i=2
					\end{array}
				\right.
			\]
			in $H_1(L_1) \oplus H_1(L_2)$. Applying Proposition~\ref{prop:DesiredDivisorOnY}, we get the desired result.
		\end{proof}
		
		Referring to Proposition~\ref{prop:VectorBundleOnReducibleSurface} and Lemma~\ref{lem: Divisor for Higher ranks}, we can assemble several exceptional vector bundles on the reducible surface $\tilde{\mathcal X_0} = \tilde X_0 \cup W_1 \cup W_2$\,(Table~\ref{table:ExcBundles_OnSingular}).
		
			\begin{center}
				\begin{tabular}{c|c|c|c}
					\raisebox{-5pt}[11pt][6pt]{}$\tilde{\mathcal X}_0$ & $\tilde X_0$ & $W_1$ & $W_2$ \\
					\hline \raisebox{-5pt}[11pt][7pt]{}$\mathcal O_{\tilde{\mathcal X}_0}$ & $\mathcal O_{\tilde X_0}$ & $\mathcal O_{W_1}$ & $\mathcal O_{W_2}$ \\
					\raisebox{-5pt}[11pt][7pt]{}$\tilde{\mathcal F}_{ij}\,{\scriptstyle (1 \leq i\neq j \leq 9)}$ & $\mathcal O_{\tilde X_0}(\iota_*(F_i - F_j))$ & $\mathcal O_{W_1}$ & $\mathcal O_{W_2}$ \\
					\raisebox{-5pt}[11pt][7pt]{}$\tilde{\mathcal C}_0$ & $\mathcal O_{\tilde X_0}(\iota_*C_0)$ & $\mathcal O_{W_1}$ & $\mathcal O_{W_2}$ \\
					\raisebox{-1pt}[11pt][7pt]{$\tilde{\mathcal K}$} & $\mathcal O_{\tilde X_0}(\iota_* K_Y)$ & $\mathcal O_{W_1}(1)$ & $\mathcal O_{W_2}(n-1)$ \\
					\raisebox{-1pt}[11pt][7pt]{$\tilde{\mathcal V}_1$} & $\mathcal O_{\tilde X_0}(\iota_*V_1)^{\oplus 2}$ & $\mathcal T_{W_1}(-1)$ & $\mathcal O_{W_2}^{\oplus 2}$ \\
					\raisebox{-1pt}[11pt][7pt]{$\tilde{\mathcal V}_2$} & $\mathcal O_{\tilde X_0}(\iota_*V_2)^{\oplus n}$ & $\mathcal O_{W_1}^{\oplus n}$ & $\mathcal \mathcal G_2$					
				\end{tabular}\vskip-0.33\baselineskip
				\nopagebreak\captionof{table}{Examples of exceptional vector bundles constructed using Proposition~\ref{prop:VectorBundleOnReducibleSurface}}\label{table:ExcBundles_OnSingular}
			\end{center}
		Standard arguments, such as \cite[p.~1181]{Hacking:ExceptionalVectorBundle}, in the deformation theory say that if an exceptional vector bundle $\tilde{\mathcal D}$ is given in the central fiber of the family $\tilde{\mathcal X}/(0 \in T)$, then it deforms uniquely in a small neighborhood of the family, \textit{i.e.} there exists a vector bundle $\mathscr D$ on $\tilde{\mathcal X}$\,(shrinking $T$ if necessary) such that $\mathscr D \big\vert_{\tilde{ \mathcal X}_0} = \tilde{\mathcal D}$.
		\begin{proposition}\label{prop:Hacking vs Topological}
			Let $\tilde {\mathcal D}$ be the exceptional line bundle on the reducible surface $\tilde{\mathcal X}_0$ obtained in Proposition~\ref{prop:VectorBundleOnReducibleSurface}. Let $\mathscr D$ be a line bundle on $\tilde{\mathcal X}$ such that $\mathscr D \big\vert_{\tilde{\mathcal X}_0} = \tilde{\mathcal D}$. Then, $\mathscr D\big\vert_{X^\gen} = \mathcal O_{X^\gen}(D^\gen)$ where $D^\gen$ is the divisor introduced in \textsection\ref{subsec:TopologyofX}.
		\end{proposition}
		\begin{proof}
			Let $\mathcal B \subset \mathcal X$ be the disjoint union of two small balls around $P_i \in \mathcal X$, and let $\tilde{\mathcal B} = \Phi^{-1}\mathcal B$.
			Using the argument in \cite[p.~1192]{Hacking:ExceptionalVectorBundle}, we observe that the class $c_1\bigl(\mathscr D\big\vert_{\tilde{\mathcal X}_t \setminus \tilde{\mathcal B}_t}\bigr) \in H^2(\tilde{\mathcal X}_t \setminus \tilde{\mathcal B}_t)$ is independent of $t$ when we identify groups $\{H^2(\tilde{\mathcal X}_t \setminus \tilde{\mathcal B}_t)\}_t$ in the natural way. For $t=0$, Poincar\'e duality on manifolds with boundaries gives a sequence of isomorphisms
			\begin{equation}
				H^2(\tilde{\mathcal X}_0 \setminus \tilde{\mathcal B}_0) \simeq H_2(X \setminus B, \partial B) \simeq H_2(X,B) \simeq H_2(X),	\label{eq: Specialization as Poincare dual}
			\end{equation}
			which convey $c_1(\mathscr D\big\vert_{\tilde{\mathcal X}_0 \setminus \tilde{\mathcal B}_0})$ to $[\pi_*D] \in H_2(X)$. As topological cycles, both $c_1(D^\gen\big\vert_{\tilde{\mathcal X}_t \setminus \tilde{\mathcal B}_t})$ and $c_1\bigl(\mathscr D\big\vert_{\tilde{\mathcal X}_t \setminus \tilde{\mathcal B}_t}\bigr)$ are obtained from $[\pi_*D\big\vert_{X \setminus B}]$ by the trivial extension, hence they coincide. The injective map $H_2(X^\gen) \to H_2(X)$ defined in Proposition~\ref{prop:Hacking_Specialization} is nothing but the natural restriction $H^2(X^\gen) \to H^2(X^\gen \setminus M)$, where the source and the target are changed by Poincar\'e duality on manifolds with boundaries. Thus, $H^2(X^\gen) \to H^2(X^\gen \setminus M)$ is injective, so $c_1(D^\gen) = c_1(\mathscr D\big\vert_{X^\gen})$. The first Chern class map $c_1 \colon \Pic X^\gen \to H^2(X^\gen,\Z)$ is an isomorphism, hence $\mathcal O_{X^\gen}(D^\gen) = \mathscr D\big\vert_{X^\gen}$.
		\end{proof}
		We finish this section by presenting an exceptional collection of length $9$ on the Dolgachev surface $X^\gen$. Note that this collection cannot generate the whole category $\D^{\rm b}(X^\gen)$.
		\begin{proposition}\label{prop:ExceptCollection_ofLengthNine}
			Let $F_{1j}^\gen\, (j>1)$ be the divisor on $X^\gen$, which arises from the deformation of $\tilde {\mathcal F}_{1j}$ along $\tilde{\mathcal X} / (0 \in T')$. Then the ordered tuple
			\[
				\bigl\langle \mathcal O_{X^\gen},\, \mathcal O_{X^\gen}(F_{12}^\gen),\,\ldots,\, \mathcal O_{X^\gen}(F_{19}^\gen) \bigr\rangle
			\]
			forms an exceptional collection in the derived category $\D^{\rm b}(X^\gen)$.
		\end{proposition}
		\begin{proof}
			By virtue of upper-semicontinuity, it suffices to prove that $H^p(\tilde{\mathcal X}_0, \tilde{\mathcal F}_{1i} \otimes \tilde{\mathcal F}_{1j}^\vee) = 0$ for $1 \leq i < j \leq 9$ and $p \geq 0$. The sequence (\ref{eq:ExactSeq_onReducibleSurface}) for $\tilde{\mathcal E} = \tilde{\mathcal F}_{1i} \otimes \tilde{\mathcal F}_{1j}^\vee$ reads
			\[
				0 \to \tilde{\mathcal F}_{1i} \otimes \tilde{\mathcal F}_{1j}^\vee \to \mathcal O_{\tilde X_0}(\iota_* ( F_j - F_i)) \oplus \mathcal O_{W_1} \oplus \mathcal O_{W_2} \to \mathcal O_{Z_1} \oplus \mathcal O_{Z_2} \to 0.
			\]
			Since $H^0(\mathcal O_{W_k}) \simeq H^0(\mathcal O_{Z_k})$ and $H^p(\mathcal O_{W_k}) = H^p (\mathcal O_{Z_k}) = 0$ for $k=1,2$ and $p > 0$, it suffices to prove that $H^p(\mathcal O_{\tilde X_0} ( \iota_*(F_j - F_i) ) ) = 0$ for all $p\geq 0$ and $i< j$. The surface $\tilde X_0$ is normal({\it cf.} \cite[p.~1178]{Hacking:ExceptionalVectorBundle}) and the divisor $F_j - F_i$ does not intersect with the exceptional locus of $\iota \colon Y \to \tilde X_0$. By Proposition~\ref{prop:CohomologyComparison_YtoX}, $H^p( \tilde X_0, \iota_*(F_j-F_i)) \simeq H^p(Y, F_j-F_i)$ for all $p \geq 0$. It remains to prove that $H^p(Y, F_j-F_i) = 0$ for $p \geq 0$. By Riemann-Roch,
			\[
				\chi(F_j-F_i) = \frac{1}{2} ( F_j-F_i \mathbin. F_j - F_i - K_Y) + 1,
			\]
			and this is zero by Lemma~\ref{lem:CanonicalofY}.  Since $(F_j\mathbin.F_j - F_i) = -1$ and $F_i \simeq \P^1$, in the short exact sequence
			\[
				0 \to \mathcal O_Y(-F_i) \to \mathcal O_Y(F_j-F_i) \to \mathcal O_{F_i}(F_j) \to 0,
			\]
			we obtain $H^0(-F_i) \simeq H^0(F_j-F_i)$. In particular, $H^0(F_j-F_i) = 0$. By Serre duality and Lemma~\ref{lem:CanonicalofY}, $H^2(F_j-F_i) = H^0(E_1 + F_i - F_j - C_2 - \ldots - E_{r+1})^*$. Similarly, since $(E_1\mathbin.E_1 + F_i - F_j - C_2 - \ldots - E_{r+1}) < 0$, $(F_i\mathbin.F_i - F_i - C_2 - \ldots - E_{r+1}) < 0$, and $E_1$, $F_j$ are rational curves, $H^0( E_1 + F_i - F_j - C_2 - \ldots - E_{r+1}) \simeq H^0(-F_j - C_2 - \ldots - E_{r+1}) = 0$. This proves that $H^2(F_j-F_i) = 0$. Finally, $\chi (F_j - F_i) = 0$ implies $H^1(F_j - F_i) =0$.
		\end{proof}
		\begin{remark}\label{rmk:ExceptCollection_SerreDuality}
			In Proposition~\ref{prop:ExceptCollection_ofLengthNine}, the trivial bundle $\mathcal O_{X^\gen}$ can be replaced by the deformation of the line bundle $\tilde{\mathcal K}^\vee$\,(Table~\ref{table:ExcBundles_OnSingular}). The strategy of the proof differs nothing. Since $\tilde{\mathcal K}^\vee$ deforms to $\mathcal O_{X^\gen}(-K_{X^\gen})$, taking dual shows that
			\[
				\bigl\langle \mathcal O_{X^\gen}(F_{21}^\gen),\,\ldots,\, \mathcal O_{X^\gen}(F_{91}^\gen) ,\, \mathcal O_{X^\gen}(K_{X^\gen}) \bigr\rangle
			\]
			is also an exceptional collection in $\D^{\rm b}(X^\gen)$. This will be used later\,(see Step~\ref{item:ProofFreePart_thm:ExceptCollection_MaxLength} in the proof of Theorem~\ref{thm:ExceptCollection_MaxLength}).
		\end{remark}
		\section{The N\'eron-Severi lattices of Dolgachev surfaces of type $(2,3)$}\label{sec:NeronSeveri}
			This section is devoted to study the simplest case, namely, the case $n=3$ and $a=1$. The surface $Y$ has the configuration as in Figure~\ref{fig:Configuration_Basic}. We cook up several divisors on $X^\gen$ according to the recipe designed below.
		\begin{recipe}\label{recipe:PicardLatticeOfDolgachev}
		Recall that $\pi \colon Y \to X$ is the contraction of $C_1, C_2, E_2$ and $\iota \colon Y \to \tilde X_0$ is the contraction of $E_2$.		
		\begin{enumerate}[label=\normalfont(\arabic{enumi})]
			\item Pick a divisor $D \in \Pic Y$ satisfying $(D.C_1) \in 2 \Z$, $(D. C_2) \in 3\Z$, and $(D. E_2) = 0$.
			\item As in Proposition~\ref{prop:VectorBundleOnReducibleSurface}, attach suitable line bundles on $W_i$\,($i=1,2$) to $\mathcal O_{\tilde X_0}(\iota_* D)$ to produce a line bundle, say $\tilde{\mathcal D}$, on $\tilde{\mathcal X}_0 = \tilde X_0 \cup W_1 \cup W_2$. It deforms to a line bundle $\mathcal O_{X^\gen}(D^\gen)$ on the Dolgachev surface $X^\gen$.
			\item Use the short exact sequence (\ref{eq:ExactSeq_onReducibleSurface}) to compute $\chi(\tilde {\mathcal D})$. By the deformation invariance of Euler characteristics, $\chi(D^\gen) = \chi(\tilde {\mathcal D})$.
			\item\label{item:recipe:CanonicalIntersection} Since the divisor $\pi_* C_0$ is away from the singularities of $X$, it is Cartier. By Lemma~\ref{lem:Intersection_withFibers}, $(C_0^\gen.D^\gen) = (C_0.D)$. Furthermore, $C_0^\gen = 6K_{X^\gen}$, thus the Riemann-Roch formula on the surface $X^\gen$ reads
			\[
				(D^\gen)^2 = \frac{1}{6}( D . C_0) + 2 \chi(\tilde {\mathcal D}) - 2.
			\]
			This computes the intersections of divisors in $X^\gen$.
		\end{enumerate}
		\end{recipe}
		By Proposition~\ref{prop:Hacking vs Topological}, $D^\gen$ is essentially determined by looking at the preimage of the cycle class $[\pi_*D] \in H_2(X)$ along the map in the sequence (\ref{eq:SpecializationSequence}). This suggests the following use of the terminology.
		\begin{definition}
			Let $D \in \Pic Y$ and $D^\gen \in \Pic X^\gen$ be as in Recipe~\ref{recipe:PicardLatticeOfDolgachev}. We call $D^\gen$ the \emph{lifting} of $D$.
		\end{definition}
		We note that this is a slight abuse of terminologies. What lifts to $D^\gen \in \Pic X^\gen$ is $\pi_* D \in \Cl X$, not $D \in \Pic Y$.

		\begin{lemma}\label{lem:EulerChar_WtdProj}
			Let $h \in H_2(W_2,\Z)$ be the hyperplane class of $W_2 = \P(1,2,1)$. For any even integer $n \in \Z$,
			\[
				\chi( \mathcal O_{W_2}(n) ) = \frac{1}{4}n(n+4) + 1.
			\]
		\end{lemma}
		\begin{proof}
			By well-known properties of weighted projective spaces, $(1 \cdot 2 \cdot 1)h^2 = 1$, $c_1(K_{W_2}) = -(1+2+1)h = -4h$, and $\mathcal O_{W_2}(2)$ is invertible. The Riemann-Roch formula for invertible sheaves\,({\it cf.} \cite[Lemma~7.1]{Hacking:ExceptionalVectorBundle}) says that $\chi(\mathcal O_{W_2}(n)) = \frac{1}{2}(  nh  \mathbin. (n+4)h) + 1 = \frac{1}{4}n(n+4) + 1$.
		\end{proof}
		\begin{lemma}\label{lem:EulerCharacteristics}
			Let $S$ be a projective normal surface with $\chi(\mathcal O_S) = 1$. Assume that all the divisors below are supported on the smooth locus of $S$. Then,
			\begin{enumerate}[ref=(\alph{enumi})]
				\item $\chi(D_1 + D_2) = \chi(D_1) + \chi(D_2) + (D_1 . D_2) - 1$;
				\item $\chi(-D) = -\chi(D) + D^2 + 2$;
				\item $\chi(-D) = p_a(D)$ where $p_a(D)$ is the arithmetic genus of $D$;
				\item $\chi(nD) = n\chi(D) + \frac{1}{2} n(n-1)D^2 - n + 1$ for all $n \in \Z$.
				\item[\rm(d$'$)] $\chi(nD) = n^2\chi(D) + \frac{1}{2}n(n-1) (K_S .D) - n^2 + 1$ for all $n \in \Z$.
			\end{enumerate}
			Assume in addition that $D$ is an integral curve with $p_a(D) = 0$. Then
			\begin{enumerate}[resume]
				\item $\chi(D) = D^2 + 2$, $\chi(-D) = 0$;
				\item $\chi(nD) = \frac{1}{2}n(n+1)D^2 + (n+1)$ for all $n \in \Z$.
			\end{enumerate}
		\end{lemma}
		\begin{proof}
			All the formula in the statement are simple variants of the Riemann-Roch formula.
		\end{proof}
		\begin{lemma}\label{lem:Intersection_withFibers}
			Let $D$, $\tilde{\mathcal D}$, $D^\gen$ be as in Recipe~\ref{recipe:PicardLatticeOfDolgachev}. Then, $(C_0.D) = (C_0^\gen . D^\gen)$.
		\end{lemma}
		\begin{proof}
			Since $C_0$ does not intersect with $C_1,C_2,E_2$, the corresponding line bundle $\tilde{\mathcal C}_0$ on $\tilde{\mathcal X}_0$ is the gluing of $\mathcal O_{\tilde X_0}(\iota_*C_0)$, $\mathcal O_{W_1}$, and $\mathcal O_{W_2}$. Thus, $(\tilde{\mathcal D} \otimes \tilde{\mathcal C}_0) \big\vert_{W_i} = \tilde{\mathcal D}\big\vert_{W_i}$ for $i=1,2$. From this and (\ref{eq:ExactSeq_onReducibleSurface}), it can be immediately shown that $\chi(\tilde{\mathcal D} \otimes \tilde{\mathcal C}_0 ) - \chi(\tilde{\mathcal D}) = \chi(D + C_0 ) - \chi (D)$. If $D$ is a principal divisor on $Y$, then the previous equation says $\chi(C_0^\gen) = \chi(\tilde{\mathcal C}_0) = \chi(C_0) = 1$. Now, using Lemma~\ref{lem:EulerCharacteristics}(a), we deduce $(C_0^\gen. D^\gen) = \chi(D^\gen + C_0^\gen) - \chi(D^\gen) = \chi(\tilde{\mathcal D} \otimes \tilde{\mathcal C}_0 ) - \chi(\tilde{\mathcal D}) = \chi(D + C_0 ) - \chi (D) = (C_0.D)$. \qedhere
		\end{proof}
		\begin{definition}
			Let $H \in \Pic \P^2$ be a line, let $p \colon Y \to \P^2$ be the blow down morphism, and let $L = p^*(2H)$ be the proper transform of a general plane conic. Then, $( L . C_1) = 6$, $(L . C_2) = 6$ and $(L. E_2) = 0$. Let $L^\gen$ be the lifting of $L$. This means that there exists a line bundle $\tilde{\mathcal L}$ on the reducible surface $\tilde{\mathcal X}_0 = \tilde X_0 \cup W_1 \cup W_2$ such that
			\[
				\tilde{\mathcal L}\big\vert_{\tilde X_0} = \mathcal O_{\tilde X_0}(\iota_* L),\quad \tilde{\mathcal L}\big\vert_{W_1} = \mathcal O_{W_1}(3),\ \text{and}\quad \tilde{\mathcal L}\big\vert_{W_2} = \mathcal O_{W_2}(4),
			\]
			which deforms to the line bundle $\mathcal O_{X^\gen}(L^\gen)$ on $X^\gen$. Let $F_{ij}^\gen \in \Pic X^\gen$ be the lifting of $F_i - F_j$, or equivalently, the divisor associated with the deformation of $\tilde{\mathcal F}_{ij}$\,(Table~\ref{table:ExcBundles_OnSingular}). We define 
			\begin{align*}
				G_i^\gen &:= - L^\gen + 10K_{X^\gen} + F_{i9}^\gen \quad \text{for }i=1,\ldots,8; \\
				G_9^\gen &:= -L^\gen + 11K_{X^\gen}.
			\end{align*}
		\end{definition}
		\begin{proposition}\label{prop:G_1to9}
			The divisors $G_1^\gen,\ldots,G_9^\gen$ satisfy the following numerical properties:
			\begin{enumerate}
				\item $\chi(G_i^\gen) = 1$ and $(G_i^\gen  . K_{X^\gen}) = -1$;
				\item for $i < j$, $\chi(G_i^\gen - G_j^\gen)=0$.
			\end{enumerate}
			In particular, $(G_i^\gen)^2 = -1$ and $(G_i^\gen. G_j^\gen) = 0$ for $1 \leq i < j \leq 9$.
		\end{proposition}
		\begin{proof}
			First, consider the case $i \leq 8$. By Recipe~\ref{recipe:PicardLatticeOfDolgachev}\ref{item:recipe:CanonicalIntersection} and $K_{X^\gen}^2 = 0$, $(K_{X^\gen} . G_i^\gen) = \frac{1}{6}(C_0 \mathbin. -L + F_i - F_9) = -1$.
			Since the alternating sum of Euler characteristics in the sequence (\ref{eq:ExactSeq_onReducibleSurface}) is zero, we get the formula
			\begin{align*}
				\chi( \tilde{\mathcal L}^\vee \otimes \tilde{\mathcal F}_{i9}) ={}& \chi(-L + F_i - F_9) + \chi(\mathcal O_{W_1}(-3)) + \chi(\mathcal O_{W_2}(-4)) \\
				&{}- \chi(\mathcal O_{Z_1}(-6)) - \chi(\mathcal O_{Z_2}(-6)).
			\end{align*}
			From this we compute $\chi( \tilde{\mathcal L}^\vee \otimes \tilde{\mathcal F}_{i9}) = 11$.
			The Riemann-Roch formula for $-L^\gen + F_{i9}^\gen = G_i^\gen - 10K_{X^\gen}$ says $(G_i^\gen - 10K_{X^\gen})^2 - (K_{X^\gen} \mathbin. G_i^\gen - K_{X^\gen} ) = 20$, hence $(G_i^\gen)^2 = -1$ and $\chi(G_i^\gen) = 1$. For $1 \leq i < j \leq 8$, $G_i - G_j = F_i - F_j$. Since $(F_i - F_j \mathbin. C_1 ) = (F_i - F_j \mathbin. C_2 ) = (F_i - F_j \mathbin. E_2 ) = 0$, the divisor $F_i - F_j$ lifts to the Cartier divisor $F_{ij}^\gen$. Hence, we can compute $\chi(G_i^\gen - G_j^\gen) = \chi(F_i - F_j) = 0$. This proves the statement for $i,j \leq 8$. The proof of the statement involving $G_9^\gen$ follows the same lines. Since $\chi(\tilde{\mathcal L}^\vee ) = 12$, $( G_9^\gen - 11K_{X^\gen} ) ^2 - (K_{X^\gen}\mathbin . G_9^\gen - 11K_{X^\gen}) = 22$. This leads to $(G_9^\gen)^2 = -1$. For $i \leq 8$,
			\begin{align*}
				\chi(G_i^\gen - G_9^\gen) ={}& \chi(F_i - F_9 - K_Y) + \chi(\mathcal O_{W_1}(-1)) + \chi(\mathcal O_{W_2}(-2)) \\
				&{} - \chi(\mathcal O_{Z_1}(-2)) - \chi(\mathcal O_{Z_2}(-3)),
			\end{align*}
			and it is immediate to see that the right hand side is zero.
		\end{proof}
		We complete our list of generators of $\Pic X^\gen$ by introducing $G_{10}^\gen$. The choice of $G_{10}^\gen$ is motivated by the proof of the step (iii)${}\Rightarrow{}$(i) in \cite[Theorem~3.1]{Vial:Exceptional_NeronSeveriLattice}.
		\begin{proposition}\label{prop:G_10}
			Let $G_{10}^\gen$ be the $\Q$-divisor $\frac{1}{3}( G_1^\gen + G_2^\gen + \ldots + G_9^\gen - K_{X^\gen})$. Then, $G_{10}^\gen$ is Cartier.
		\end{proposition}
		\begin{proof}
			Since
			\[
				\sum_{i=1}^9 G_i^\gen - K_{X^\gen} = - 9L^\gen + 90 K_{X^\gen} + \sum_{i=1}^8 F_{i9}^\gen,
			\]
			it suffices to prove that $\sum\limits_{i=1}^8 F_{i9}^\gen = 3D^\gen$ for some $D^\gen \in \Pic X^\gen$. Let $p \colon Y \to \P^2$ be the blowing up morphism and let $H$ be a line in $\P^2$. Since $K_Y = p^* (-3H) + F_1 + F_2 + \ldots + F_9 + E_1 + E_2 + 2E_3$, $K_Y - E_1 - E_2 - 2E_3 = p^*(-3H) + F_1 + \ldots + F_9 = -C_0$, so $F_1 + \ldots + F_9 = 3p^*H - C_0$. Consider the divisor $p^*H - 3F_9$ in $Y$. Clearly, the intersections of $(p^*H - 3F_9)$ with $C_1$, $C_2$, $E_2$ are all zero, hence $p^*H - 3F_9$ lifts to a Cartier divisor $(p^*H-3F_9)^\gen$ in $X^\gen$. Since 
			\begin{align*}
				\sum_{i=1}^8 (F_i - F_9) &= \sum_{i=1}^9 F_i - 9F_9 \\
				&= 3(p^*H - 3F_9) - C_0
			\end{align*}
			and $C_0$ lifts to $6K_{X^\gen}$, $D^\gen := (p^*H - 3F_9)^\gen - 2K_{X^\gen}$ satisfies $\sum_{i=1}^8F_{i9}^\gen = 3D^\gen$.
		\end{proof}
		Combining the propositions \ref{prop:G_1to9} and \ref{prop:G_10}, we obtain:\nopagebreak
		\begin{theorem}\label{thm:Picard_ofGeneralFiber}
			The intersection matrix of divisors $\{G_i^\gen\}_{i=1}^{10}$ is
			\begin{equation}
				\Bigl( (G_i^\gen . G_j^\gen ) \Bigr)_{1 \leq i,j \leq 10} =  \left[
				\begin{array}{cccc}
					-1  & \cdots & 0 & 0 \\
					\vdots & \ddots & \vdots & \vdots \\
					0 &  \cdots & -1 & 0 \\
					0 &  \cdots & 0 & 1
				\end{array}
				\right]\raisebox{-2\baselineskip}[0pt][0pt]{.}		\label{eq:IntersectionMatrix}
			\end{equation}
			In particular, the set $G:=\{G_i^\gen\}_{i=1}^{10}$ forms a $\Z$-basis of the N\'eron-Severi lattice $\op{NS}(X^\gen)$. By \cite[p.~137]{Dolgachev:AlgebraicSurfaces}, $\Pic X^\gen$ is torsion-free, thus $G$ forms a $\Z$-basis for $\Pic X^\gen$.
		\end{theorem}
		\begin{proof}
			We claim that the set of divisors $\{G_i^\gen\}_{i=1}^{10}$ generates the N\'eron-Severi lattice. By Hodge index theorem, there is a $\Z$-basis for $\op{NS}(X^\gen)$, say $\alpha = \{\alpha_i\}_{i=1}^{10}$, such that the intersection matrix with respect to $\{\alpha_i\}_{i=1}^{10}$ is the same as (\ref{eq:IntersectionMatrix}). Let $A = ( a_{ij} )_{1 \leq i,j \leq 10}$ be the integral matrix determined by
			\[
				G_i^\gen = \sum_{j=1}^{10} a_{ij} \alpha_j.
			\]
			Given $v \in \op{NS}(X^\gen)$, let $[v]_\alpha$ be the column matrix of coordinates with respect to the basis $\alpha$. Then, $[G_i^\gen]_\alpha = Ae_i$ where $e_i$ is the $i$th column vector. For $1 \leq i,j \leq 10$,
			\[
				(G_i^\gen . G_j^\gen) = (A e_i)^t E (A e_j)
			\]
			where $E$ is the intersection matrix with respect to the basis $\alpha$. The above equation implies that the intersection matrix with respect to the set $G$ is $A^{\rm t} E A$. Since the intersection matrices with respect to both $G$ and $\alpha$ are the same, $E = A^{\rm t} E A$. This implies that $1 = \det (A^{\rm t}A) = (\det A)^2$, hence $A$ is invertible over $\Z$. This proves that $G$ is a $\Z$-basis of $\op{NS}(X^\gen)$. The last statement on the Picard group follows immediately.
		\end{proof}
		We close the section with the summary of divisors on $X^\gen$.
		\begin{summary}\label{summary:Divisors_onX^g}
			Recall that $Y$ is the rational elliptic surface in Section~\ref{sec:Construction}, $p \colon Y \to \P^2$ is the blow down morphism, $H \in \Pic \P^2$ is a hyperplane divisor, and $\pi \colon Y \to X$ is the contraction of $C_1,\,C_2,\,E_2$. Then,
			\begin{enumerate}[label=\normalfont(\arabic{enumi})]
				\item $F_{ij}^\gen$\,($1\leq i,j\leq 9$) is the lifting of $F_i - F_j$;
				\item $(p^*H-3F_9)^\gen$ is the lifting of $p^*H - 3F_9$;
				\item $L^\gen$ is the lifting of $p^*(2H)$;
				\item $G_i^\gen = - L^\gen + 10 K_{X^\gen} + F_{i9}^\gen$ for $i=1,\ldots,8$;
				\item $G_9^\gen = -L^\gen + 11K_{X^\gen}$;
				\item $G_{10}^\gen = -3L^\gen + (p^*H - 3F_9)^\gen + 28K_{X^\gen}$.
			\end{enumerate}
		\end{summary}
		\section{Exceptional collections of maximal length on Dolgachev surfaces of type $(2,3)$}\label{sec:ExcepCollectMaxLength}
			\subsection{Exceptional collections of maximal length}%
			We continue to study the case $(n,a)=(3,1)$. Throughout this section, we will prove that there exists an exceptional collection of maximal length in $\D^{\rm b}(X^\gen)$. Proving exceptionality of a given collection usually consists of numerous cohomology computations, so we begin by introducing some computational machineries.
		\begin{lemma}\label{lem:DummyBundle}
			The liftings $C_1^\gen$, $(2C_2+E_2)^\gen$ exist and they are the zero divisors in $X^\gen$.
		\end{lemma}
		\begin{proof}
			Let $\tilde{\mathcal C}_1$ be the gluing of line bundles $\mathcal O_{\tilde X_0}(\iota_* C_1)$, $\mathcal O_{W_1}(-2)$, and $\mathcal O_{W_2}$, and let $\mathcal O_{X^\gen}(C_1^\gen)$ be its deformation. It is immediate to see that $\chi(C_1^\gen) = 1$ and $\chi(-C_1^\gen)=1$. By Riemann-Roch formula, $(C_1^\gen)^2 = (C_1^\gen . K_{X^\gen}) = 0$. For $i \leq 8$,
			\begin{align*}
				\chi(G_i^\gen - 10K_{X^\gen} - C_1^\gen ) ={}& \chi( \tilde{\mathcal L}^\vee \otimes \tilde{\mathcal F}_{i9} \otimes \tilde{\mathcal C}_1^\vee ) \\
				={}& \chi(-L + F_i - F_9 - C_1) + \chi(\mathcal O_{W_1}(-1)) + \chi(\mathcal O_{W_2}(-4)) \\
				&{} - \chi(\mathcal O_{Z_1}(-2)) - \chi(\mathcal O_{Z_2}(-6)),
			\end{align*}
			which yields $\chi(G_i^\gen - 10K_{X^\gen} - C_1^\gen )=11$. By the Riemann-Roch, $(G_i^\gen - 10K_{X^\gen}- C_1^\gen)^2 - (K_{X^\gen} \mathbin . G_i^\gen - 10K_{X^\gen} - C_1^\gen) = 2\chi(G_i^\gen -10K_{X^\gen} - C_1^\gen)-2 = 20$. The left hand side is $-2(G_i^\gen . C_1^\gen) + 20$, thus $(G_i^\gen . C_1^\gen)  =0$. Since $(C_1^\gen . K_{X^\gen}) = 0$ and $3G_{10}^\gen = G_1^\gen + \ldots + G_9^\gen - K_{X^\gen}$, $(G_{10}^\gen . C_1^\gen) = 0$. Hence, $C_1^\gen$ is numerically trivial by Theorem~\ref{thm:Picard_ofGeneralFiber}. This shows that $C_1^\gen$ is trivial since there is no torsion in $\Pic X^\gen$. Exactly the same argument holds for the lifting of $2C_2 + E_2$.
		\end{proof}
		\begin{example}\label{eg:DivisorVaries_onSingular}
			Since $C_0$ lifts to $6K_{X^\gen}$, $2E_1 = C_0 - C_1$ lifts to $6K_{X^\gen}$. Thus $E_1$ lifts to $3K_{X^\gen}$. Similarly, $C_2 + E_2 + E_3$ lifts to $2K_{X^\gen}$. Hence, $K_Y = E_1 - C_2 - E_2 - E_3$ lifts to $3K_{X^\gen} - 2K_{X^\gen} = K_{X^\gen}$. Also, $(E_2 + 2E_3) - E_1$ lifts to $K_{X^\gen}$, whereas $K_Y$ and $(E_2+2E_3)-E_1$ are different in $\Pic Y$. These are essentially due to Lemma~\ref{lem:DummyBundle}. For instance, we have
			\begin{align*}
				(E_2 + 2E_3) - E_1 - K_Y &= - 2 E_1 + C_2 + 2E_2 + 3E_3 \\
				&= -C_1,
			\end{align*}
			thus $(E_2 + 2E_3)^\gen - E_1^\gen - K_{X^\gen} = -C_1^\gen = 0$.
		\end{example}
		As Example~\ref{eg:DivisorVaries_onSingular} presents, there are free spaces to choose $D \in \Pic Y$ given a fixed divisor $D^\gen \in \Pic X^\gen$. The following lemma gives a direction to choose $D$. Note that the lemma requires assumptions on $(D.C_1)$ and $(D.C_2)$, but Lemma~\ref{lem:DummyBundle} provides the way to adjust those numbers.
		\begin{lemma}\label{lem:H0Computation}
			Let $D$ be a divisor in $Y$ such that $(D.C_1) = 2d_1 \in 2\Z$, $(D.C_2) = 3d_2 \in 3\Z$, and $(D . E_2) = 0$. %
			Let $D^\gen$ be the lifting of $D$. Then,
			\[
				h^0(X^\gen, D^\gen) \leq h^0(Y,D) + h^0(\mathcal O_{W_1}(d_1)) + h^0(\mathcal O_{W_2}(2d_2)) - h^0(\mathcal O_{Z_1}(2d_1)) - h^0(\mathcal O_{Z_2}(3d_3)).
			\]
			In particular, if $d_1,d_2 \leq 1$, then $h^0(X^\gen, D^\gen) \leq h^0(Y,D)$.
		\end{lemma}
		\begin{proof}
			Since $(D. E_2) = 0$, we have $H^p(\tilde X_0,\iota_*D) \simeq H^p(Y,D)$ for all $p \geq 0$\,(Proposition~\ref{prop:CohomologyComparison_YtoX}). Recall that there exists a short exact sequence (introduced in (\ref{eq:ExactSeq_onReducibleSurface}))
			\begin{equation}
				0 \to \tilde{\mathcal D} \to \mathcal O_{\tilde X_0}(\iota_*D) \oplus \mathcal O_{W_1}(d_1) \oplus \mathcal O_{W_2}(2d_2) \to \mathcal O_{Z_1}(2d_1) \oplus \mathcal O_{Z_2}(3d_2) \to 0, \label{eq:ExactSeq_onReducibleSurface_SimplerVer}
			\end{equation}
			where $\tilde{\mathcal D}$ is the line bundle constructed in Proposition~\ref{prop:VectorBundleOnReducibleSurface}, and the notations $W_i$, $Z_i$ are explained in (\ref{eq:SecondWtdBlowupExceptional}). We first claim the following: if $d_1,d_2 \leq 1$, then the maps $H^0(\mathcal O_{W_1}(d_1)) \to H^0(\mathcal O_{Z_1}(2d_1))$ and $H^0(\mathcal O_{W_2}(2d_2)) \to H^0(\mathcal O_{Z_2}(3d_2))$ are isomorphisms. Only the nontrivial cases are $d_1 = 1$ and $d_2 = 1$. Since $Z_1$ is a smooth conic in $W_1 = \P^2$, there is a short exact sequence
			\[
				0 \to \mathcal O_{W_1}(-1) \to \mathcal O_{W_1}(1) \to \mathcal O_{Z_1}(2) \to 0.
			\]
			All the cohomology groups of $\mathcal O_{W_1}(-1)$ vanish, so $H^p( \mathcal O_{W_1}(1)) \simeq H^p(\mathcal O_{Z_1}(2))$ for all $p \geq 0$. In the case $d_2 = 1$, we consider
			\[
				 0 \to \mathcal I_{Z_2}(2) \to \mathcal O_{W_2}(2) \to \mathcal O_{Z_2}(3) \to 0,
			\]
			where $\mathcal I_{Z_2} \subset \mathcal O_{W_2}$ is the ideal sheaf of the closed subscheme $Z_2 = (xy  = z ^3 ) \subset \P_{x,y,z}(1,2,1)$. The ideal $(xy - z^3)$ does not contain any nonzero homogeneous element of degree $2$, so $H^0(\mathcal I_{Z_2}(2)) = 0$. This shows that $H^0( \mathcal O_{W_2}(2)) \to H^0(\mathcal O_{Z_2}(3))$ is injective. Furthermore, $H^0(\mathcal O_{W_2}(2))$ is generated by $x^2, xz, z^2, y$, hence $h^0(\mathcal O_{W_2}(2)) = h^0(\mathcal O_{Z_3}(3)) = 4$. This proves that $H^0(\mathcal O_{W_2}(2)) \simeq H^0(\mathcal O_{Z_3}(3))$, as desired. If $d_1,d_2>1$, it is clear that $H^0(\mathcal O_{W_1}(d_1)) \to H^0(\mathcal O_{Z_1}(2d_1))$ and $H^0(\mathcal O_{W_2}(2d_2)) \to H^0(\mathcal O_{Z_2}(3d_2))$ are surjective.
			
			The cohomology long exact sequence of (\ref{eq:ExactSeq_onReducibleSurface_SimplerVer}) begins with
			\begin{align*}
				0 \to H^0(\tilde{\mathcal D}) \to H^0(\iota_*D) \oplus H^0( \mathcal O_{W_1}(d_1) ) \oplus H^0(\mathcal O_{W_2}(2d_2)) \\
				\to H^0(\mathcal O_{Z_1}(2d_1)) \oplus H^0( \mathcal O_{Z_2}(3d_2)).\qquad\qquad
			\end{align*}
			By the previous arguments, the last map is surjective. Indeed, the image of $(0, s_1, s_2) \in H^0(\iota_*D) \oplus H^0( \mathcal O_{W_1}(d_1) ) \oplus H^0(\mathcal O_{W_2}(2d_2))$ is $(-s_1\big\vert_{Z_1}, -s_2\big\vert_{Z_2})$. The upper-semicontinuity of cohomologies establishes the inequality in the statement. \qedhere
		\end{proof}
		By \cite[Theorem~3.1]{Vial:Exceptional_NeronSeveriLattice}, it can be shown that the collection (\ref{eq:ExcColl_MaxLength}) in the theorem below is a numerically exceptional collection. Our aim is to prove that (\ref{eq:ExcColl_MaxLength}) is indeed an exceptional collection in $\D^{\rm b}(X^\gen)$. Before proceed to the theorem, we introduce one terminology.
		\begin{definition}
			During the construction of $Y$, the node of $p_*C_2$ is blown up twice, which corresponds to one of the two tangent directions\footnote{For example, the nodal curve $y^2 = x^3+x^2z$ has two tangent directions $y=\pm x$ at $(0,0,1)$.} at the node of $p_*C_2$. We refer to the tangent direction corresponding to the second blow up as the \emph{distinguished tangent direction} at the node of $p_*C_2$.
		\end{definition}
		\begin{theorem}\label{thm:ExceptCollection_MaxLength}
			Suppose $X^\gen$ is originated from a cubic pencil $\lvert \lambda p_* C_1 + \mu p_*C_2 \rvert$ which is generated by two general plane nodal cubics. Let $G_1^\gen,\ldots,G_{10}^\gen$ be as in Summary~\ref{summary:Divisors_onX^g}, let $G_0^\gen$ be the zero divisor, and let $G_{11}^\gen = 2G_{10}^\gen$. For notational simplicity, we denote the rank of $\Ext^p(G_i^\gen, G_j^\gen)(=H^p(-G_i^\gen+G_j^\gen))$ by $h^p_{ij}$. The values of $h_{ij}^p$ are described below. For example, the triple of ($G_9^\gen$-row, $G_{10}^\gen$-column), which is $(0\ 0\ 2)$, means that $(h^0_{9,10},\, h^1_{9,10},\, h^2_{9,10}) = (0,0,2)$.
			\begin{equation}
				\scalebox{0.9}{$
					\begin{array}{c|ccccc}
						& G_0^\gen & G_{1 \leq i \leq 8}^\gen & G_9^\gen & G_{10}^\gen & G_{11}^\gen \\[2pt] \hline%
						G_0^\gen & 1\,0\,0 & 0\,0\,1 &  0\,0\,1 & 0\,0\,3 & 0\,0\,6 \\
						G_{1 \leq i \leq 8}^\gen & & 1\,0\,0 & & 0\,0\,2 & 0\,0\,5\\
						G_9^\gen & & & 1\,0\,0 & 0\,0\,2 & 0\,0\,5		\\
						G_{10}^\gen & & & & 1\,0\,0 & 0\,0\,3 \\
						G_{11}^\gen & & & & & 1\,0\,0
					\end{array}
				$}\label{eq:thm:ExceptCollection_MaxLength}
			\end{equation}
			The symbol $G_{1 \leq i \leq 8}^\gen$ means $G_i^\gen$ for each $i=1,2,\ldots,8$. The blanks stand for $0\,0\,0$, and $h^p_{ij} = 0$ for all $p$ and $1 \leq i\neq j \leq 8$. In particular, the collection
			\begin{equation}
				\big\langle \mathcal O_{X^\gen}(G_0^\gen),\ \mathcal O_{X^\gen}(G_1^\gen),\ \ldots,\ \mathcal O_{X^\gen}(G_{10}^\gen),\ \mathcal O_{X^\gen}(G_{11}^\gen) \big\rangle \label{eq:ExcColl_MaxLength}
			\end{equation}
			is an exceptional collection of length $12$ in $\D^{\rm b}(X^\gen)$.
			\end{theorem}
		\begin{proof}
			Recall that (Summary~\ref{summary:Divisors_onX^g})
			\begin{align*}
				G_i^\gen &= -L^\gen + F_{i9}^\gen + 10K_{X^\gen},\ i=1,\ldots,8; \\
				G_9^\gen &= -L^\gen + 11K_{X^\gen}; \\
				G_{10}^\gen &= -3L^\gen + (p^*H - 3F_9)^\gen + 28K_{X^\gen}; \\
				G_{11}^\gen &= -6L^\gen + 2(p^*H - 3F_9)^\gen + 56K_{X^\gen}.
			\end{align*}
			The proof consists of numerous cohomology vanishings for which we divide into several steps. Note that we can always evaluate $\chi(-G_i^\gen + G_j^\gen) = \sum_p (-1)^p h^p_{ij}$, thus it suffices to compute only two (mostly $h^0$ and $h^2$) of $\{h^p_{ij} : p=0,1,2\}$.

			In the first part of the proof, we deduce the following using numerical methods.
			\begin{equation}
				\scalebox{0.9}{$
				\begin{array}{c|ccccc}
					& G_0^\gen & G_{1 \leq i \leq 8}^\gen & G_9^\gen & G_{10}^\gen & G_{11}^\gen \\[2pt] \hline%
					G_0^\gen 				& 1\,0\,0 & 0\,0\,1 & 0\,0\,1 & \scriptstyle\chi=3 & \scriptstyle\chi=6  \\
					G_{1 \leq i \leq 8}^\gen & 0\,0\,0 & 1\,0\,0 & 0\,0\,0 & 0\,0\,2 & \scriptstyle\chi=5 \\
					G_9^\gen				 &\scriptstyle\chi=0 & 0\,0\,0 & 1\,0\,0 & 0\,0\,2 &	\scriptstyle\chi=5	\\
					G_{10}^\gen 			&\scriptstyle\chi=0 &\scriptstyle\chi=0 &\scriptstyle\chi=0 & 1\,0\,0 & \scriptstyle\chi=3 \\
					G_{11}^\gen 			&\scriptstyle\chi=0 &\scriptstyle\chi=0 &\scriptstyle\chi=0 &\scriptstyle\chi=0 & 1\,0\,0
				\end{array}
				$}\label{eq:HumanPart_thm:ExceptCollection_MaxLength}
			\end{equation}
			The slots with $\chi=d$ means $\chi(-G_i^\gen+G_j^\gen)=\sum_{p}(-1)^ph^p_{ij} =d$. For those slots, we do not compute each $h^p_{ij}$ for the moment. In the end, they will be completed through another approach.
			\begin{enumerate}[fullwidth, itemsep=5pt minus 3pt, label=\bf{}Step~\arabic{enumi}., ref=\arabic{enumi}]
				\item\label{item:NumericalStep_thm:ExceptCollection_MaxLength} As explained above, the collection (\ref{eq:ExcColl_MaxLength}) is numerically exceptional, hence $\chi(-G_i^\gen + G_j^\gen) = \sum_p (-1)^p h^p_{ij}=0$ for all $0 \leq j < i \leq 11$. Furthermore, the surface $X^\gen$ is minimal, thus $K_{X^\gen}$ is nef. It follows that $h^0(D^\gen) = 0$ if $D^\gen$ is $K_{X^\gen}$-negative, and $h^2(D^\gen)=0$ if $D^\gen$ is $K_{X^\gen}$-positive. Since
				\[
					(K_{X^\gen} . G_i^\gen) = \left\{
						\begin{array}{ll}
							-1 & i \leq 9 \\
							-3 & i = 10 \\
							-6 & i = 11,
						\end{array}
					\right.
				\]
				these already enforce a number of cohomologies to be zero. Indeed, all the numbers in the following list are zero:
				\[
					\{ h^0_{0i} \}_{i \leq 11},\ \{ h^0_{i,10}, h^0_{i,11} \}_{i \leq 9},\ \{ h^2_{i0} \}_{i \leq 11},\ \{ h^2_{10,i}, h^2_{11,i} \}_{i \leq 9}.
				\]
				Also, since $G_{11}^\gen = 2G_{10}^\gen$, $h^0_{10,11} = h^0_{0,10} = 0$ and $h^2_{11,10} = h^2_{10,0} = 0$.
				\item\label{item:ProofFreePart_thm:ExceptCollection_MaxLength} If $1 \leq j \neq i \leq 8$, then $-G_i^\gen + G_j^\gen$ can be realized as the lifting of $-F_i + F_j$. Hence,
				\[
					\bigl\langle \mathcal O_{X^\gen}(G_1^\gen),\ \ldots,\ \mathcal O_{X^\gen}(G_8^\gen) \bigr\rangle
				\]
				is an exceptional collection by Proposition~\ref{prop:ExceptCollection_ofLengthNine}. This proves that $h^p_{ij}=0$ for all $p \geq 0$ and $1 \leq i \neq j \leq 8$. Also, $-G_9^\gen + G_i^\gen = -K_{X^\gen} + F_{i9}^\gen$ for $1 \leq i \leq 8$. Remark~\ref{rmk:ExceptCollection_SerreDuality} shows that $h^p_{9i} = h^p(-K_{X^\gen} + F_{i9}^\gen)=0$ for $p \geq 0$ and $1 \leq i \leq 8$. Furthermore, by Serre duality, $h^p_{i9} = h^{2-p}(F_{i9}^\gen)=0$ for all $p \geq 0$ and $1 \leq i \leq 8$.
				\item\label{item:Strategy_HumanPart_thm:ExceptCollection_MaxLength} We verify (\ref{eq:HumanPart_thm:ExceptCollection_MaxLength}) using the following strategy:
				\begin{enumerate}
					\item If we want to compute $h^0_{ij}$, then pick $D_{ij}^\gen := -G_i^\gen + G_j^\gen$. If the aim is to evaluate $h^2_{ij}$, then take $D_{ij}^\gen := K_{X^\gen} + G_i^\gen - G_j^\gen$, so that $h^2_{ij} = h^0(D_{ij}^\gen)$ by Serre duality.
					\item Express $D^\gen_{ij}$ in terms of $L^\gen$, $(p^*H - 3F_9)^\gen$, $F_{i9}^\gen$, and $K_{X^\gen}$. Via Summary~\ref{summary:Divisors_onX^g}, we can translate $L^\gen$, $(p^*H - 3F_9)^\gen$, $F_{i9}^\gen$ into the divisors on $Y$. Further, we have $6K_{X^\gen} = C_0^\gen$, $3K_{X^\gen} = E_1^\gen$, and $2K_{X^\gen} = (C_2+E_2+E_3)^\gen$, thus an arbitrary integer multiple of $K_{X^\gen}$ also can be translated into divisors on $Y$. Together with these translations, use Lemma~\ref{lem:DummyBundle} to find a Cartier divisor $D_{ij}$ on $Y$ which lifts to $D_{ij}^\gen$, and which satisfies $(D_{ij}.C_1) \leq 2$, $(D_{ij}.C_2) \leq 3$, $(D_{ij}.E_2)=0$.
					\item Compute an upper bound of $h^0(D_{ij})$. Then by Lemma~\ref{lem:H0Computation}, $h^0(D_{ij}^\gen) \leq h^0(D_{ij})$.
					\item In any occasions, we will find that the upper bound obtained in (3) coincides with $\chi(-G_i^\gen + G_j^\gen)$. Also, at least one of $\{ h^0_{ij}, h^2_{ij}\}$ is zero by Step~\ref{item:NumericalStep_thm:ExceptCollection_MaxLength}. From this we deduce $h^0(D_{ij}^\gen)\geq{}$(the upper bound obtained in (3)), hence the equality holds. Consequently, the numbers $\{h^p_{ij} : p=0,1,2\}$ are evaluated.
				\end{enumerate}
				\item We follow the strategy in Step~\ref{item:Strategy_HumanPart_thm:ExceptCollection_MaxLength} to verify (\ref{eq:HumanPart_thm:ExceptCollection_MaxLength}). Let $i \in \{1,\ldots,8\}$. To verify $h^0_{i0}=0$, we take $D_{i0}^\gen = -G_i^\gen = L^\gen - F_{i9}^\gen - 10K_{X^\gen}$. Translation into the divisors on $Y$ gives:
				\[
					D_{i0}' = p^*(2H) + F_9-F_i - 2C_0 + (C_2+E_2+E_3)
				\]
				Since $(D_{i0}'.C_1) = 6$ and $(D_{i0}' . C_2) = 3$, we replace the divisor $D_{i0}'$ by $D_{i0} := D_{i0}' + C_1$ so that the condition $(D_{i0}.C_1) \leq 2$ is fulfilled. Now, $h^0(D_{i0})=0$ by Dictionary~\ref{dictionary:H0Computations}\ref{dictionary:-G_i}, thus $h^0_{i0} \leq h^0(D_{i0}) = 0$ by Lemma~\ref{lem:H0Computation}. Finally, $\chi(-G_i^\gen)=0$ and $h^2_{i0}=0$\,(Step~\ref{item:NumericalStep_thm:ExceptCollection_MaxLength}), hence $h^1_{i0}=0$.
					
				We repeat this routine to the following divisors:
				\begin{align*}
					D_{0i} &= p^*(2H) + F_9 - F_i - C_0 + C_1 - E_1 + (2C_2+E_2); \\
					D_{09} &= p^*(2H) - 2C_0 + C_1 + (C_2 + E_2 + E_3); \\
					D_{i,10} &= p^*(3H) + 2F_9 + F_i - 2C_0 + 2C_1 - E_1 - (C_2 + E_2 + E_3) + 2(2C_2+E_2); \\
					D_{9,10} &= p^*(3H) + 3F_9 - 3C_0 + 3C_1 + (C_2+E_2+E_3) + (2C_2+E_2).
				\end{align*}
				Together with Dictionary~\ref{dictionary:H0Computations}\hyperref[dictionary:Nonvanish_K-G_i]{(2--5)}, all the slots of (\ref{eq:HumanPart_thm:ExceptCollection_MaxLength}) are verified.
				\item\label{item:M2Part_thm:ExceptCollection_MaxLength} It is difficult to complete (\ref{eq:thm:ExceptCollection_MaxLength}) using the numerical argument\,(see for example, Remark~\ref{rmk:Configuration_andCohomology}). We introduce another plan to overcome these difficulties.
				\begin{enumerate}
					\item Take $D_{ij}^\gen \in \Pic X^\gen$ and $D_{ij} \in \Pic Y$ as in Step~\hyperref[item:Strategy_HumanPart_thm:ExceptCollection_MaxLength]{3(1--2)}. We may assume $(D_{ij}.C_1) \in \{0,2\}$ and $(D_{ij}.C_2) \in \{-3,0,3\}$. If $(D_{ij}.C_2) = -3$, then $(D_{ij} - C_2 \mathbin. E_2) = -1$, thus $h^0(D_{ij}) = h^0(D_{ij} - C_2 - E_2)$. Hence, we replace $D_{ij}$ by $D_{ij} - C_2 -E_2$ if $(D.C_2) = -3$. In some occasions, we have $(D_{ij} . F_9) = -1$. We make further replacement $D_{ij} \mapsto D_{ij}-F_9$ for those cases.
					\item Rewrite $D_{ij}$ in terms of the $\Z$-basis $\{p^*H, F_1,\ldots, F_9, E_1,E_2,E_3\}$ so that $D_{ij}$ is expressed in the following form:
					\[
						D_{ij} = p^*(dH) - \bigl( \text{sum of exceptional curves of }p \colon Y \to \P^2\bigr).
					\]
					\item\label{item:PlaneCurveExistence_thm:ExceptCollection_MaxLength} If $h^0(D_{ij}) > 0$, then we consider an effective divisor $D$ which is linearly equivalent to $D_{ij}$. Then, $p_*D$ is the plane curve of degree $d$ which satisfied the conditions imposed by the exceptional part of $D_{ij}$. Let $\mathcal I_{\rm C} \subset \mathcal O_{\P^2}$ be the ideal sheaf associated with the imposed conditions on $p_*D$. Then the curve $p_*D$ contributes to the number $h^0(\mathcal O_{\P^2}(d) \otimes \mathcal I_{\rm C})$. Indeed, this number gives an upper bound of $h^0(D_{ij})$\,(it is clear that if $D'$ is an effective divisor linearly equivalent to $D$ such that $p_*D$ and $p_*D'$ coincide as plane curves, then $D$ and $D'$ must be the same curve in $Y$).
					\item As in Step~\hyperref[item:Strategy_HumanPart_thm:ExceptCollection_MaxLength]{3(4)}, we will see that all the upper bounds $h^0(D_{ij})$ coincide with the numerical invariants $\chi(-G_i^\gen + G_j^\gen)$. Thus, the upper bounds $h^0(D_{ij})$ obtained in (3) determine $\{h^p_{ij} : p=0,1,2\}$ precisely.
				\end{enumerate}
					\item As explained in Remark~\ref{rmk:Configuration_andCohomology}, the value $h^0(D_{ij})$ might depend on the configuration of $p_*C_1$ and $p_*C_2$. However, for general $p_*C_1 = (h_1=0)$, $p_*C_2 = (h_2=0)$, the minimum value of $h^0(D_{ij})$ is attained. This can be observed in the following way. Let $h = \sum_{\alpha} a_\alpha \mathbf{x}^\alpha$ be a homogeneous equation of degree $d$, where the sum is taken over the $3$-tuples $\alpha = (\alpha_x, \alpha_y, \alpha_z)$ with $\alpha_x + \alpha_y + \alpha_z = d$ and $\mathbf{x}^\alpha = x^{\alpha_x} y^{\alpha_y} z^{\alpha_z}$. Then the ideal $\mathcal I_{\rm C}$ imposes linear relations on $\{a_\alpha\}_\alpha$, thus we get a linear system, or equivalently, a matrix $M$, with the variables $\{a_\alpha\}_\alpha$. After perturbing $h_1$ and $h_2$, the rank of $M$ would not decrease since \{$\op{rank} M \geq r_0$\} is an open condition for any fixed $r_0$. From this we conclude: if $h^0(D_{ij}) \leq r$ for at least one pair of $p_*C_1$ and $p_*C_2$, then $h^0(D_{ij}) \leq r$ for general $p_*C_1$ and $p_*C_2$.
					\item\label{item:M2Configuration_thm:ExceptCollection_MaxLength} Let $h_1 = (y-z)^2z - x^3 - x^2z$ and $h_2 = x^3 - 2xy^2 + 2xyz + y^2z$. These equations define plane nodal cubics such that
					\begin{enumerate}
						\item $p_*C_1$ has the node at $[0,1,1]$, and $p_*C_2$ has the node at $[0,0,1]$;
						\item $p_*C_2$ has two tangent directions ($y=0$ and $y=-2x$) at nodes;
						\item $p_*C_1 \cap p_*C_2$ contains two $\Q$-rational points, namely $[0,1,0]$ and $[-1,1,1]$.
					\end{enumerate}
					We take $y=0$ as the distinguished tangent direction at the node of $p_*C_2$, and take $p_*F_9 = [0,1,0]$, $p_*F_8 = [-1,1,1]$. The ideals in Table~\ref{table:Ideal_ofConditions_thm:ExceptCollection_MaxLength} are the building blocks of the ideal $\mathcal I_{\rm C}$ introduced in Step~\hyperref[item:PlaneCurveExistence_thm:ExceptCollection_MaxLength]{5(3)}.
					\[
						\begin{array}{c|c|c|c}
							\text{symbol} & \text{ideal form} & \text{ideal sheaf at the\,...} & \text{divisor on }Y \\ \hline
							\mathcal I_{E_1} & (x,y-z) & \text{node of }p_*C_1 & -E_1 \\
							\mathcal I_{E_2+E_3} & (x,y) & \text{node of }p_*C_2 & -(E_2+E_3) \\
							\mathcal I_{E_2+2E_3} & (x^2,y) & %
								\begin{tabular}{c}
									\footnotesize distinguished tangent \\[-5pt]
									\footnotesize at the node of $p_*C_2$
								\end{tabular} %
							& -(E_2+2E_3) \\
							\mathcal J_9 & (h_1,h_2) &  \text{nine base points} & - \sum_{i\leq9} F_i \\
							\mathcal J_7 & \scriptstyle \mathcal J_9 {\textstyle/} (x+z,y-z)(x,z) & \text{seven base points} & -\sum_{i\leq7} F_i \\
							\mathcal J_8 & (x+z,y-z)\mathcal J_7 & \text{eight base points} & -\sum_{i\leq8} F_i
						\end{array}
					\]\nopagebreak\vskip-\baselineskip\captionof{table}{The ideals associated with the exceptional divisors}\label{table:Ideal_ofConditions_thm:ExceptCollection_MaxLength}\vskip+0.33\baselineskip
					Note that the nine base points contain $[0,1,0]$ and $[-1,1,1]$, thus there exists an ideal $\mathcal J_7$ such that $\mathcal J_9 = (x+z,y-z)(x,z) \mathcal J_7$.
					\item We sketch the proof of $h^p_{10,9}=h^p(-G_{10}^\gen + G_9^\gen)=0$, which illustrates several subtleties. Since $h^2_{10,9}=0$ by Step~\ref{item:NumericalStep_thm:ExceptCollection_MaxLength}, we only have to prove $h^0_{10,9}=0$. Thus, we take $D_{10,9}^\gen := -G_{10}^\gen + G_9^\gen$. As in Step~\hyperref[item:Strategy_HumanPart_thm:ExceptCollection_MaxLength]{3(2)}, take $D_{10,9}' = p^*(3H) + 3F_9 - 2C_0 + 2C_1 - E_1 - (C_2+E_2+E_3) + 2(2C_2+E_2)$. We have $(D_{10,9}'.C_2)=-3$, and $(D_{10,9}'-C_2-E_2 \mathbin. F_9) = -1$. Let $D_{10,9}:= D_{10,9}' - C_2-E_2-F_9$. Then, $h^0(D_{10,9}) = h^0(D_{10,9}') \geq h^0(D_{10,9}^\gen)$. As in Step~\hyperref[item:M2Part_thm:ExceptCollection_MaxLength]{5(2)}, the divisor $D_{10,9}$ can be rewritten as
					\[
						D_{10,9} = p^*(9H) - 2 \sum_{i=1}^8 F_i - 5E_1 - 4E_2 - 7E_3.
					\]
					Since $\mathcal I_{E_2+E_3}^2$ imposes more conditions than $\mathcal I_{E_2+2E_3}$, the ideal of (minimal) conditions corresponding to $-4E_2 - 7E_3$ is $\mathcal I_{E_2+E_3} \cdot \mathcal I_{E_2+2E_3}^3$. Thus, the plane curve $p_*D_{10,0}$ corresponds to a nonzero section of
					\[
						H^0(\mathcal O_{\P^2}(9) \otimes \mathcal J_8^2 \cdot \mathcal I_{E_1}^5 \cdot \mathcal I_{E_2+E_3} \cdot \mathcal I_{E_2+2E_3}^3 ).
					\]
					Using Macaulay2, we find that the rank of this group is zero. This can be found in \texttt{ExcColl\_Dolgachev.m2}\,\cite{ChoLee:Macaulay2}. In a similar way, we obtain the following table (be aware of the difference with (\ref{eq:thm:ExceptCollection_MaxLength})).
					\begin{equation}
						\scalebox{0.9}{$
						\begin{array}{c|ccccc}
							& G_0^\gen & G_8^\gen & G_9^\gen & G_{10}^\gen & G_{11}^\gen \\[2pt] \hline%
							G_0^\gen & 1\,0\,0 & 0\,0\,1 &  0\,0\,1 & 0\,0\,3 & 0\,0\,6 \\
							G_8^\gen & & 1\,0\,0 & & 0\,0\,2 & 0\,0\,5\\
							G_9^\gen & & & 1\,0\,0 & 0\,0\,2 & 0\,0\,5		\\
							G_{10}^\gen & & & & 1\,0\,0 & 0\,0\,3 \\
							G_{11}^\gen & & & & & 1\,0\,0
						\end{array}
						$}
						\label{eq:M2Computation_thm:ExceptCollection_MaxLength}
					\end{equation}
					Table~\ref{table: Macaualy2 Computations} gives a short summary on the computations done in \texttt{ExcColl\_Dolgachev.m2}\,\cite{ChoLee:Macaulay2}.
					\[
						\scalebox{0.9}{$
						\begin{array}{c|c|l}
							(i,j) & \text{result} & \multicolumn{1}{c}{\text{choice of }D_{ij}} \\ \hline
							(9,0) & h^0_{9,0}=0 & p^*(5H) - \sum_{i\leq9} F_i - 3E_1 - 2E_2 - 4E_3 \\
							(10,0) & h^0_{10,0}=0 & p^*(14H) - 3\sum_{i\leq8} F_i - 8E_1 - 6E_2-11E_3 \\
							(10,8) & h^0_{10,8}=0 & p^*(9H) - 2\sum_{i\leq7} F_i - F_8 - 6E_1 - 3E_2 - 6E_3 \\
							(10,9) & h^0_{10,9}=0 & p^*(9H) - 2\sum_{i\leq8} F_i - 5E_1 - 4E_2 - 7E_3 \\
							(11,0) & h^0_{11,0}=0 & p^*(31H) - 7\sum_{i\leq8} F_i - F_9 - 18E_1 - 11E_2 - 22E_3 \\
							(11,8) & h^0_{11,8}=0 & p^*(26H) - 6\sum_{i\leq7} F_i - 5F_8 - F_9 - 14E_1 - 10E_2 - 20E_3 \\
							(11,9) & h^0_{11,9}=0 & p^*(26H) - 6\sum_{i\leq8} F_i - 15E_1 - 9E_2 - 18E_3 \\[3pt] \hline
							(0,10) & h^2_{0,10}=3 & p^*(17H) - 4\sum_{i\leq8} F_i - F_9 - 9E_1 - 6E_2 - 12E_3 \\
							(0,11) & h^2_{0,11}=6 & p^*(31H) - 7\sum_{i\leq8} F_i - F_9 - 17E_1 - 12E_2 - 23E_3 \\
							(8,11) & h^2_{8,11}=5 & p^*(26H) - 6\sum_{i\leq7} F_i - 5F_8 - F_9 - 15E_1 - 9E_2 - 18E_3 \\
							(9,11) & h^2_{9,11}=5 & p^*(26H) - 6\sum_{i\leq8} F_i - 14E_1 - 10E_2 - 19E_3
						\end{array}
						$}
					\]\nopagebreak\vskip-\baselineskip\captionof{table}{Summary of the Macaulay2 computations}\label{table: Macaualy2 Computations}\vskip+0.33\baselineskip
					Note that the numbers $h_{11,10}^p$ and $h_{10,11}^p$ are computed freely; indeed, $-G_{11}^\gen + G_{10}^\gen = -G_{10}^\gen$, thus $h^p_{11,10} = h^p_{10,0}$ and $h^p_{10,11}=h^p_{0,10}$. Finally, perturb the cubics $p_*C_1$ and $p_*C_2$ so that (\ref{eq:M2Computation_thm:ExceptCollection_MaxLength}) remains valid and Lemma~\ref{lem:BasePtPermutation} is applicable. Then, (\ref{eq:thm:ExceptCollection_MaxLength}) is verified immediately. \qedhere
			\end{enumerate}
		\end{proof}
		\begin{remark}\label{rmk:Configuration_andCohomology}
			Assume that the nodal curves $p_*C_1$, $p_*C_2$ are in a special position so that the node of $p_*C_1$ is located on the distinguished tangent line at the node of $p_*C_2$. Then, the proper transform $\ell$ of the unique line through the nodes of $p_*C_1$ and $p_*C_2$ has the following divisor expression:
			\[
				\ell = p^*H - E_1 - (E_2 + 2E_3).
			\]
			In particular, the divisor $D_{90} = p^*(5H) - \sum_{i\leq9} F_i - 3E_1 - 2(E_2+2E_3)$ is linearly equivalent to $2 \ell + C_1 + E_1$, thus $h^0(D_{90}) > 0$. Consequently, for this particular configuration of $p_*C_1$ and $p_*C_2$, we cannot prove $h^0_{90} = 0$ using upper-semicontinuity. However, the numerical method (Step~\ref{item:Strategy_HumanPart_thm:ExceptCollection_MaxLength} in the proof of the previous theorem) cannot detect such variances originated from the position of nodal cubics, hence it cannot be applied to the proof of $h^0_{90}=0$.
		\end{remark}
		The following lemma, used in the end of the proof of Theorem~\ref{thm:ExceptCollection_MaxLength}, illustrates the symmetric nature of $F_1,\ldots,F_8$.
		\begin{lemma}\label{lem:BasePtPermutation}
			Assume that $X^\gen$ is originated from a cubic pencil generated by two general plane nodal cubics $p_*C_1$ and $p_*C_2$. Let $D \in \Pic Y$ be a divisor on the rational elliptic surface $Y$. Assume that in the expression of $D$ in terms of $\Z$-basis $\{p^*H, F_1,\ldots, F_9, E_1, E_2, E_3\}$, the coefficients of $F_1,\ldots, F_8$ are the same. Then, $h^p(D+F_i) = h^p(D+F_j)$ for any $p \geq 0$ and $1 \leq i,j \leq 8$.
		\end{lemma}
		\begin{proof}
			Since $\op{Aut}\P^2 = \op{PGL}(3,\C)$ sends arbitrary 4 points\,(of which any three are not colinear) to arbitrary 4 points\,(of which any three are not colinear), we may assume the following.
			\begin{enumerate}
				\item\label{item:lem:BasePtPermutation_DistinguishedBasePt} The base point $p_*F_9$ is $\Q$-rational.
				\item The nodes of $p_*C_1$ and $p_*C_2$ are $\Q$-rational.
				\item The distinguished tangent direction at the node of $p_*C_2$ is defined over $\Q$.
			\end{enumerate}
			Now, let $K$ be the extension field over $\Q$ which is generated by the coefficients of the cubic forms defining $p_*C_1$, $p_*C_2$. Since $p_*C_1$, $p_*C_2$ are general, we may assume the following:
			\begin{enumerate}[resume]
				\item\label{item:lem:BasePtPermutation_affineBasePts} The base points $p_*F_1,\ldots,p_*F_9$ are contained in the affine space $(z \neq 0) \subset \P^2_{x,y,z}$.
				\item\label{item:lem:BasePtPermutation_Resultant} Let $h_i \in K[x,y,z]$ be the defining equation of $p_*C_i$, and let $\op{res}(h_1,h_2;x)$ be the resultant of $h_1(x,y,1)$, $h_2(x,y,1)$ regarded as elements in $(K[x])[y]$. The irreducible factorization of $\op{res}(h_1,h_2;x)$ over $K$ consists of a linear form and an irreducible polynomial, say $H_x$, of degree $8$. The same holds for $\op{res}(h_1,h_2;y)$, {\it i.e.} $\op{res}(h_1,h_2;y) = (y-c) H_y$ for an irreducible polynomial $H_y \in K[y]$ of degree $8$. We assume further that $H_x \neq H_y$ up to multiplication by $K^\times$.

			\end{enumerate}
			The last condition has the following interpretation. Let $p_* F_i = [\alpha_i,\beta_i,1] \in \P^2$ for $\alpha_i,\beta_i \in \C$ and $i=1,\ldots,9$. The resultant $\op{res}(h_1,h_2;x) \in K[x]$ is the polynomial having $\{\alpha_i\}_{i=1}^9$ as the solutions. By the conditions \ref{item:lem:BasePtPermutation_DistinguishedBasePt}, $\alpha_9 \in \Q$, so a linear factor must appear in $\op{res}(h_1,h_2;x)$. Hence, \ref{item:lem:BasePtPermutation_Resultant} implies that $\alpha_1,\ldots,\alpha_8$ are Galois conjugate over $K$, which should be true for general $p_*C_1$, $p_*C_2$. The same is assumed to be true for $\beta_1,\ldots,\beta_8$, and the final sentence says that $\{\alpha_1,\ldots,\alpha_8\} \neq \{\beta_1,\ldots,\beta_8\}$.
			
			
			Let $\tau \in \op{Aut}(\C/K)$ be a field automorphism fixing $K$, and mapping $\alpha_i$ to $\alpha_j$\,($1 \leq i,j \leq 8$). Then $\tau$ induces an automorphism of $\P^2$ which fixes $p_*C_1$ and $p_*C_2$. It follows that $[\alpha_j, \tau(\beta_i), 1]$ is one of the eight base points $\{p_*F_i\}_{i=1}^8$. Since $H_x$ and $H_y$ are different up to multiplication by $K^\times$, there is no point of the form $[\alpha_j, \beta_k, 1]$ in the set $\{p_*F_i\}_{i=1}^8$ except when $k=j$. It follows that $\tau(\beta_i) = \beta_j$. Let $\tau_Y \colon Y \to Y$ be the automorphism induced by $\tau$. According to the assumptions \ref{item:lem:BasePtPermutation_DistinguishedBasePt}--\ref{item:lem:BasePtPermutation_Resultant}, it satisfies the following properties:
			\begin{enumerate}[label=(\arabic{enumi})]
				\item $\tau_Y$ fixes $F_9, E_1, E_2, E_3$;
				\item $\tau_Y$ permutes $F_1,\ldots,F_8$;
				\item $\tau_Y$ maps $F_i$ to $F_j$.
			\end{enumerate}
			Furthermore, since the coefficients of $F_1,\ldots,F_8$ are the same in the expression of $D$, $\tau_Y$ fixes $D$. It follows that $\tau_Y^* \colon \Pic Y \to \Pic Y$ maps $D+F_j$ to $D+F_i$. In particular, $H^p(D+F_j) = H^p(\tau_Y^*(D+F_i)) \simeq H^p( D+F_i)$ for any $1 \leq i,j \leq 8$. \qedhere
		\end{proof}
		\subsection{Incompleteness of the collection}\label{subsec:Incompleteness} Let $\mathcal A \subset \D^{\rm b}(X^\gen)$ be the orthogonal subcategory
		\[
			\bigl\langle \mathcal O_{X^\gen}(G_0^\gen),\ \mathcal O_{X^\gen}(G_1^\gen),\ \ldots,\ \mathcal O_{X^\gen}(G_{11}^\gen) \bigr\rangle^\perp,
		\]	
		so that there exists a semiorthogonal decomposition
		\[
			\D^{\rm b}(X^\gen) = \bigl\langle \mathcal A,\ \mathcal O_{X^\gen}(G_0^\gen),\ \mathcal O_{X^\gen}(G_1^\gen),\ \ldots,\ \mathcal O_{X^\gen}(G_{11}^\gen) \bigr\rangle.
		\]
		We will prove that $K_0(\mathcal A) = 0$, $\op{HH}_\bullet(\mathcal A) = 0$, but $\mathcal A\not\simeq 0$. Such a category is called a \emph{phantom} category. To give a proof, we claim that the \emph{pseudoheight} of the collection~(\ref{eq:ExcColl_MaxLength}) is at least $2$. Once we achieve the claim, \cite[Corollary~4.6]{Kuznetsov:Height} implies that $\op{HH}^0(\mathcal A) \simeq \op{HH}^0(X^\gen) = \C$, thus $\mathcal A\not\simeq 0$.
		\begin{definition}\ 
			\begin{enumerate}
				\item Let $E_1,E_2$ be objects in $\D^{\rm b}(X^\gen)$. The \emph{relative height} $e(E_1,E_2)$ is the minimum of the set
				\[
					\{ p : \Hom(E_1,E_2[p]) \neq 0 \} \cup \{ \infty \}.
				\]
				\item Let $\langle F_0,\ldots,F_m\rangle$ be an exceptional collection in $\D^{\rm b}(X^\gen)$. The \emph{anticanonical pseudoheight} is defined by
				\[
					\op{ph}_{\rm ac}(F_0,\ldots,F_m) = \min \Bigl ( \sum_{i=1}^p e(F_{a_{i-1}}, F_{a_i}) + e(F_{a_p} , F_{a_0} \otimes \mathcal O_{X^\gen}(-K_{X^\gen})) - p \Bigr),
				\]
				where the minimum is taken over all possible tuples $0 \leq a_0 < \ldots < a_p \leq m$.
			\end{enumerate}
		\end{definition}
		The pseudoheight is given by the formula $\op{ph}(F_0,\ldots,F_m) = \op{ph}_{\rm ac}(F_0,\ldots,F_m) + \dim X^\gen$, thus it suffices to prove that $\op{ph}_{\rm ac}(G_0^\gen,\ldots,G_{11}^\gen) \geq 0$.
		\begin{corollary}\label{cor:Phantom}
			In the semiorthogonal decomposition
			\[
				\D^{\rm b}(X^\gen) = \bigl \langle \mathcal A,\ \mathcal O_{X^\gen}(G_0^\gen),\ \ldots,\ \mathcal O_{X^\gen}(G_{11}^\gen)\bigr\rangle,
			\]
			we have $K_0(\mathcal A) = 0$ and $\op{HH}_\bullet(\mathcal A)=0$. Also, $\op{ph}_{\rm ac}(G_0^\gen,\ldots,G_{11}^\gen) = 2$, thus the restriction map $\op{HH}^p(X^\gen) \to \op{HH}^p(\mathcal A)$ is an isomorphism for $p \leq 2$ and is a monomorphism for $p=3$. In particular, $\op{HH}^0(\mathcal A) \simeq \C$.
		\end{corollary}
		\begin{proof}
			Since $\kappa(X^\gen) = 1$, the Bloch conjecture holds for $X^\gen$\,\cite[\textsection11.1.3]{Voisin:HodgeTheory2}. Thus the Grothendieck group $K_0(X^\gen)$ is a free abelian group of rank $12$\,(see for {\it e.g.} \cite[Lemma~2.7]{GalkinShinder:Beauville}). Furthermore, Hochschild-Kostant-Rosenberg isomorphism for Hochschild homology says
			\[
				\op{HH}_k(X^\gen) \simeq \bigoplus_{q-p=k} H^{p,q}(X^\gen),
			\]
			hence, $\op{HH}_\bullet(X^\gen) \simeq \C^{\oplus 12}$. It is well-known that $K_0$ and $\op{HH}_\bullet$ are the additive invariants with respect to semiorthogonal decompositions, thus $K_0(X^\gen) \simeq K_0(\mathcal A) \oplus K_0({}^\perp\mathcal A)$, and $\op{HH}_\bullet(X^\gen) = \op{HH}_\bullet(\mathcal A) \oplus \op{HH}_\bullet({}^\perp \mathcal A)$.\footnote{By definition of $\mathcal A$, ${}^\perp \mathcal A$ is the smallest full triangulated subcategory containing the collection (\ref{eq:ExcColl_MaxLength}) in Theorem~\ref{thm:ExceptCollection_MaxLength}.} If $E$ is an exceptional vector bundle, then $\D^{\rm b}(\langle E\rangle ) \simeq \D^{\rm b}(\Spec \C)$ as $\C$-linear triangulated categories, thus $K_0({}^\perp\mathcal A) \simeq \Z^{\oplus 12}$ and $\op{HH}_\bullet({}^\perp\mathcal A)\simeq \C^{\oplus12}$. It follows that $K_0(\mathcal A) = 0$ and $\op{HH}_\bullet(\mathcal A)=0$.

			Assume the chain $0\leq a_0 < \ldots < a_p \leq 11$ has length $p=0$. Then, $e(G_{a_0}^\gen,G_{a_0}^\gen-K_{X^\gen}) = 2$ since $\dim \Ext_{X^\gen}^p(G_i^\gen, G_i^\gen - K_{X^\gen}) = h^p(-K_{X^\gen}) = 1$ for $p=2$ and $0$ otherwise. For any $0 \leq j < i \leq 11$, 
			\[
				e(G_j^\gen, G_i^\gen) =\left\{
					\begin{array}{ll}
						\infty & \text{if } 1 \leq j < i \leq 9 \\
						2 & \text{otherwise}
					\end{array}
				\right.
			\]
			by Theorem~\ref{thm:ExceptCollection_MaxLength}. Also, it is easy to see that $H^0(G_i^\gen, G_j^\gen - K_{X^\gen}) = 0$ for $i > j$, thus for any chain $0 \leq a_0 < \ldots< a_p \leq 11$,
			\[
				e(G_{a_0}^\gen, G_{a_1}^\gen) + \ldots + e(G_{a_{p-1}}^\gen, G_{a_p}^\gen) + e(G_{a_p}^\gen, G_{a_0}^\gen - K_{X^\gen}) - p \geq 2p + 1 - p,
			\]
			which shows that the value of the left hand side is at least $2$ for any chain of length${}>0$. It follows that $\op{ph}_{\rm ac}(G_0^\gen,\ldots,G_{11}^\gen) =2$. The statements about $\op{HH}^\bullet$ immediately follows by \cite[Corollary~4.6]{Kuznetsov:Height}. \qedhere
		\end{proof}
		\subsection{Cohomology computations}
		We present Dictionary~\ref{dictionary:H0Computations} of cohomology computations that appeared in the proof of Theorem~\ref{thm:ExceptCollection_MaxLength}. It needs the divisors illustrated in Figure~\ref{fig:Configuration_Basic}, together with one more curve, which did not appear in Figure~\ref{fig:Configuration_Basic}. Let $\ell$ be the proper transform of the unique line in $\P^2$ passing through the nodes of $p_* C_1$ and $p_* C_2$. In the divisor form,
		\[
			\ell = p^*H - E_1 - (E_2 + E_3).
		\]
		Due to the divisor forms
		\[
			\begin{array}{r@{}l}
				C_1 &{}=p^*(3H) - 2E_1 - \sum_{i=1}^9F_i, \\
				C_2 &{}=p^*(3H) - (2E_2 + 3E_3) - \sum_{i=1}^9 F_i,\ \text{and} \\
				C_0 &{}=p^*(3H) - \sum_{i=1}^9 F_i,
			\end{array}
		\]
		it is straightforward to write down the intersections involving $\ell$:
		\[
			\begin{array}{  c | c | c | c | c | c | c | c | c | c }
					& p^*H & F_i & C_0 & C_1 & E_1 & C_2 & E_2 & E_3 & \ell \\
				\hline%
				\ell & 1 & 0 & 3 & 1 & 1 & 1 & 1 & 0 & -1
			\end{array}
		\]
		\begin{dictionary}\label{dictionary:H0Computations}
			For each of the following Cartier divisors on $Y$, we give upper bounds of $h^0$. The main strategy is the following. We take smooth rational curves $A_1,\ldots, A_r$, and consider the exact sequence
			\[
				0 \to H^0(D - S_i) \to H^0(D - S_{i-1}) \to H^0(\mathcal O_{A_i}( D - S_{i-1}) ),
			\]
			where $S_i = \sum_{j\leq i} A_j$. This gives the inequality $h^0(D - S_{i-1}) \leq h^0(D - S_i) + h^0( (D - S_{i-1})\big\vert_{A_i})$. Inductively, we get
			\begin{equation}
				h^0(D) \leq h^0(D - S_r) + \sum_{i=1}^{r-1} h^0( ( D - S_i)\big\vert_{A_{i+1}}). \label{eq:dictionary:UpperBound}
			\end{equation}
			In what follows, we take $A_1,\ldots,A_r$ carefully so that $h^0(D- S_r)=0$, and that the values $h^0(D - S_{i-1}\big\vert_{A_i})$ are as small as possible. In each item in the dictionary, we first present the target divisor $D$ and the bound of $h^0(D)$. After then, we give a list of smooth rational curves in the following format:
			\[
				A_1,\ A_2,\ \ldots,\ A_i\textsuperscript{(\checkmark)},\ \ldots\ , A_r.
			\]
			The symbol $(\checkmark)$ indicates the situation when $(D - S_{i-1} \mathbin. A_i) = 0$, the case in which the right hand side of (\ref{eq:dictionary:UpperBound}) increases by $1$. The curves without symbols indicate the situations in which $(D - S_{i-1} \mathbin . A_i) < 0$, so that $A_i$ does not contribute to the bound of $h^0(D)$. We conclude by showing that $D - S_r$ is not an effective divisor. The upper bound of $h^0(D)$ will be given by the number of $(\checkmark)$'s in the list. Since all of these calculations are routine, we omit the details. From now on, $i$ is any number between $1,2,\ldots,8$.
			\begin{enumerate}[label=\normalfont(\arabic{enumi}), itemsep=7pt plus 5pt minus 0pt]
				\item\label{dictionary:-G_i} $D=p^*(2H) + F_9 - F_i - 2C_0 + C_1 + C_2 + E_2 + E_3$ \hfill $h^0(D)=0$ \\
				The following is the list of curves $A_1,\ldots,A_r$\,(the order is important): $F_9,\, \ell,\, E_2,\, \ell$. The resulting divisor is 
				\[
					D - A_1 - \ldots - A_r = p^*(2H) - F_i - 2C_0 + C_1 + C_2 + E_3 - 2\ell.
				\]
				Since $\ell = p^*H - E_1 - (E_2 + E_3)$ and $C_0 = C_1 + 2E_1 = C_2 + 2E_2 + 3E_3$, $D - A_1 - \ldots - A_r = -F_i$. It follows that $H^0(D) \simeq H^0( - F_i) = 0$.
				\item\label{dictionary:Nonvanish_K-G_i}  $D = p^*(2H) + F_9 - F_i - C_0 + C_1 -E_1 + 2C_2 + E_2.$ \hfill $h^0(D) \leq 1$ \\
				Rule out $C_2,\,E_2,\,\ell\textsuperscript{(\checkmark)},\,C_1,\,F_9,\,C_2,\,\ell,\,E_1$. The resulting divisor is $p^*(2H) - F_i - C_0 - 2E_1 - 2\ell = -F_i - C_2 - E_3$. Since there is only one checkmark, $h^0(D) \leq h^0(-F_i - C_2 - E_3) + 1 = 1$.
				\item\label{dictionary:Nonvanish_K-G_9} $D = p^*(2H) - 2C_0 + C_1 + C_2 + E_2 + E_3$ \hfill $h^0(D) \leq 1$ \\
				Rule out $\ell,\, E_2,\, \ell,\, C_2\textsuperscript{(\checkmark)}$. The remaining part is $ p^*(2H) - 2C_0 + C_1 + E_3 - 2\ell = - C_2$, thus $h^0(D) \leq 1$.
				\item\label{dictionary:Nonvanish_K+G_i-G_10} $D = p^*(3H) + 2F_9 + F_i - 2C_0 + 2C_1 - E_1 + 3C_2 + E_2 - E_3$ \hfill $h^0(D) \leq 2$ \\
				The following is the list of divisors that we have to remove:
				\[
					C_2,\ E_2,\ \ell\textsuperscript{(\checkmark)},\ E_2,\ F_9\textsuperscript{(\checkmark)},\ C_2,\ E_2,\ \ell,\ C_1,\ F_9,\ F_i,\ \ell.
				\]
				The remaining part is $p^*(3H) - 2C_0 + C_1 - E_1 + C_2 - E_2 - E_3 - 3\ell = -E_3$, thus $h^0(D) \leq 2$.
				\item\label{dictionary:Nonvanish_K+G_9-G_10} $D = p^*(3H) + 3F_9 - 3C_0 + 3C_1 + 3C_2 + 2E_2 + E_3$ \hfill $h^0(D) \leq 2$ \\
				Rule out the following curves:
				\[
					F_9\textsuperscript{(\checkmark)},\ C_1,\ C_2,\ E_2,\ F_9,\ \ell,\ E_2\textsuperscript{(\checkmark)},\ \ell,\ C_2,\ \ell,\ E_2,\ E_3,\ F_9,\ C_1,\ E_1.
				\]
				The remaining part is $p^*(3H) - 3C_0 + C_1 - E_1 + C_2 - E_2 -3\ell = -C_0$, thus $h^0(D) \leq 2$.
			\end{enumerate}
		\end{dictionary}
	\section{Appendix}\label{sec:Appendix}
	\subsection{A brief review on Hacking's construction.}\label{subsec:HackingConstruction}
		Let $n>a>0$ be coprime integers, let $X$ be a projective normal surface with quotient singularities, and let $(P \in X)$ be a $T_1$-singularity of type $(0 \in \A^2 / \frac{1}{n^2}(1,na-1))$. Suppose there exists a one parameter deformation $\mathcal X / ( 0 \in T)$ of $X$ over a smooth curve germ $(0 \in T)$ such that $(P \in \mathcal X) / (0 \in T)$ is a $\Q$-Gorenstein smoothing of $(P \in X)$.
		\begin{proposition}[{\cite[\textsection3]{Hacking:ExceptionalVectorBundle}}]\label{prop:HackingWtdBlup}
			Take the base extension $(0 \in T') \to (0 \in T)$ of ramification index $a$, and let $\mathcal X'$ be the pull back along the extension. Then, there exists a proper birational morphism $\Phi \colon \tilde{\mathcal X} \to \mathcal X'$ satisfying the following properties.
			\begin{enumerate}
				\item The exceptional fiber $W = \Phi^{-1}(P)$ is isomorphic to the projective normal surface
				\[
					(xy = z^n + t^a) \subset \P_{x,y,z,t}(1,na-1,a,n).
				\]
				\item The morphism $\Phi$ is an isomorphism outside $W$.
				\item\label{item:prop:HackingWtdBlup} The central fiber $\tilde{\mathcal X}_0 = \Phi^{-1}(\mathcal X'_0)$ is reduced and has two irreducible components: $\tilde X_0$ the proper transform of $X$, and $W$. The intersection $Z:=\tilde X_0 \cap W$ is a smooth rational curve given by $(t=0)$ in $W$. Furthermore, the surface $\tilde X_0$ can be obtained in the following way: take a minimal resolution $Y \to X$ of $(P \in X)$, and let $E_1,\ldots,E_r$ be the chain of exceptional curves arranged in such a way that $(E_i . E_{i+1})=1$ and $(E_r^2) = -2$. Then the contraction of $E_2,\ldots,E_r$ defines $\tilde X_0$. Clearly, $E_1$ maps isomorphically onto $Z$ along the contraction $Y \to \tilde X_0$.
			\end{enumerate}
		\end{proposition}
		\begin{proposition}[{\cite[Proposition~5.1]{Hacking:ExceptionalVectorBundle}}]\label{prop:Hacking_BundleG}
			There exists an exceptional vector bundle $G$ of rank $n$ on $W$ such that $G \big\vert_{Z} \simeq \mathcal O_Z(1)^{\oplus n}$.
		\end{proposition}
		\begin{remark}\label{rmk:SimplestSingularCase}
			Note that in the decomposition $\tilde{\mathcal X}_0 = \tilde X_0 \cup W$, the surface $W$ is completely determined by the type of singularity $(P \in X)$, whereas $\tilde X_0$ reflects the global geometry of $X$. In some circumstances, $W$ and $G$ have explicit descriptions. 
			\begin{enumerate}
				\item Suppose $a=1$. In $\P_{x,y,z,t}(1,n-1,1,n)$, we have $W_2 =( xy = z^n + t)$ and $Z_2 = (xy=z^n, t=0)$ by Proposition~\ref{prop:HackingWtdBlup}. The projection map $\P_{x,y,z,t}(1,n-1,1,n) \dashrightarrow \P_{x,y,z}(1,n-1,1)$ sends $W_2$ isomorphically onto $\P_{x,y,z}$, thus we get
				\[
					W_2 \simeq \P_{x,y,z}(1,n-1,1),\quad\text{and}\quad Z_2 \simeq (xy=z^n) \subset \P_{x,y,z}(1,n-1,1).
				\]
				\item Suppose $(n,a) = (2,1)$, then it can be shown (by following the proof of Proposition~\ref{prop:Hacking_BundleG}) that $W = \P_{x,y,z}^2$, $G = \mathcal T_{\P^2}(-1)$ where $\mathcal T_{\P^2} = (\Omega_{\P^2}^1)^\vee$ is the tangent sheaf of the plane. Moreover, the smooth rational curve $Z = \tilde X_0 \cap W$ is embedded as a smooth conic $(xy = z^2)$ in $W$.
			\end{enumerate}
		\end{remark}
		The final proposition presents how to obtain an exceptional vector bundle on the general fiber of the smoothing.
		\begin{proposition}[{\cite[\textsection4]{Hacking:ExceptionalVectorBundle}}]\label{prop:HackingDeformingBundles}
			Let $X^\gen$ be the general fiber of the deformation $\mathcal X / (0 \in T)$, and assume $H^2(\mathcal O_{X^\gen}) = H^1(X^\gen,\Z) = 0$.\footnote{Since quotient singularities are Du Bois, we have $H^1(\mathcal O_X) = H^2(\mathcal O_X) = 0$. ({\it cf.} \cite[Lem.~4.1]{Hacking:ExceptionalVectorBundle})} Let $G$ be the exceptional vector bundle on $W$ in Proposition~\ref{prop:Hacking_BundleG}. Suppose there exists a Weil divisor $D \in \Cl X$ such that $D$ does not pass through the singular points of $X$ except $P$, the proper transform $D' \subset \tilde X_0$ of $X$ satisfies $(D'. Z) = 1$, and $\op{Supp} D' \subset \tilde X_0 \setminus \op{Sing} \tilde X_0$. Then the vector bundles $\mathcal O_{\tilde X_0}(D')^{\oplus n}$ and $G$ glue along $\mathcal O_Z(1)^{\oplus n}$ to produce an exceptional vector bundle $\tilde E$ on $\tilde{\mathcal X}_0$. Furthermore, the vector bundle $\tilde E$ deforms uniquely to an exceptional vector bundle $\tilde{\mathcal E}$ on $\tilde{\mathcal X}$. Restriction $\tilde{\mathcal E}\big\vert_{X^\gen}$ to the general fiber is an exceptional vector bundle on $X^\gen$ of rank $n$.
		\end{proposition} 
\bigskip
\footnotesize
\noindent{\bf Acknowledgments.}
	The first author thanks to Kyoung-Seog Lee for helpful comments on derived categories. He also thanks to Alexander Kuznetsov and Pawel Sosna for giving explanation on the technique of height used in Section \ref{subsec:Incompleteness}. The second author thanks to Fabrizio Catanese and Ilya Karzhemanov for useful remarks. The authors would like to appreciate many valuable comments from the anonymous referee.
	This work is supported by Global Ph.D. Fellowship Program through the National Research Foundation of Korea(NRF) funded by the Ministry of Education(No.2013H1A2A1033339)\,(to Y.C.), and is partially supported by the NRF of Korea funded by the Korean government(MSIP)(No.2013006431)\,(to Y.L.).
\bibliography{manuscript}

\end{document}